\documentclass[12pt,leqno,oneside]{amsart}
\usepackage{mathrsfs,dsfont}
\usepackage{amsmath,amstext,amsthm,amssymb}
\usepackage{charter}
\usepackage{typearea}
\usepackage{pdfsync}
\usepackage{tikz} 
\usepackage{wrapfig}
\usepackage{hyperref}
\usepackage{mathtools}

\usepackage[width=6.4in,height=8.5in]
{geometry}

\newtheorem{Theorem}{Theorem}[section]
\newtheorem{Proposition}[Theorem]{Proposition}
\newtheorem{Lemma}[Theorem]{Lemma}
\newtheorem{Corollary}[Theorem]{Corollary}
\theoremstyle{definition}

\newtheorem{Remark}[Theorem]{Remark}

\def\re{{\mathbf {Re\,}}}

\newcommand{\db}{\overline\partial}

\DeclareMathOperator{\ric}{Ric}

\DeclareMathOperator{\supp}{supp}

\DeclareMathOperator{\dist}{dist}

\DeclarePairedDelimiter\floor{\lfloor}{\rfloor}

\newcommand{\cali}[1]{\mathscr{#1}}
\newcommand{\cO}{\cali{O}} \newcommand{\cI}{\cali{I}}
\newcommand{\cM}{\cali{M}}
\newcommand{\cC}{\cali{C}}
\newcommand{\field}[1]{\mathbb{#1}}

\newcommand{\R}{\field{R}}
\newcommand{\C}{\field{C}}
\newcommand{\N}{\field{N}}

\newcommand{\D}{\field{D}}

\newcommand{\comment}[1]{}

\begin{document}

\title[Bergman kernel asymptotics for singular metrics
on punctured Riemann surfaces]{Bergman kernel asymptotics for singular metrics
on punctured Riemann surfaces}
\author{Dan Coman}
\thanks{D.\ Coman is partially supported by the NSF Grant DMS-1300157
and DMS-1700011}
\address{Department of Mathematics, Syracuse University, Syracuse, NY 13244-1150, USA}
\email{dcoman@syr.edu}
\author{Semyon Klevtsov}
\address{Universit{\"a}t zu K{\"o}ln, Institut f{\"u}r Theoretische Physik,
Z{\"u}lpicher Str.\ 77, 50937 K{\"o}ln, \newline
    \mbox{\quad}\,Deutschland}
\email{sam.klevtsov@gmail.com}
\thanks{{S.\ Klevtsov is partially supported by the German Excellence 
Initiative at the University of Cologne, grants DFG-grant ZI513/2-1, NSh-1500.2014.2,
RFBR 17-01-00585 and CRC/TR 183}}
\author{George Marinescu}
\address{Universit{\"a}t zu K{\"o}ln, Mathematisches Institut, Weyertal 86-90, 50931 K{\"o}ln, 
Deutschland   \newline
    \mbox{\quad}
Institute of Mathematics `Simion Stoilow', Romanian Academy, Bucharest, Romania}
\email{gmarines@math.uni-koeln.de}
\thanks{G.\ Marinescu is partially supported by DFG funded project CRC/TRR 191
and gratefully acknowledges the support of Syracuse University,
where part of this paper was written.}
\thanks{The authors were partially funded through the Institutional Strategy 
of the University of Cologne within the German Excellence Initiative.
They gratefully acknowledge support from the Simons Center for Geometry and Physics, 
Stony Brook University, at which some of the research for this paper was performed}
\subjclass[2010]{Primary 32L10; Secondary 32A60, 32C20, 32U40, 81Q50.}
\keywords{Bergman kernel function, singular Hermitian metric, quantum Hall effect}

\date{April 22, 2017}

\pagestyle{myheadings}

\begin{abstract} 
We consider singular metrics on a punctured Riemann surface and on a line bundle 
and study the behavior 
of the Bergman kernel in the neighborhood of the punctures. 
The results have an interpretation in terms of the asymptotic profile 
of the density of states function of the lowest Landau level in quantum Hall effect.
\end{abstract}

\maketitle
\tableofcontents

\section{Introduction}\label{S:intro}
The purpose of this paper is to study the behavior of the Bergman kernel function
of a Hermitian holomorphic line bundle over a punctured Riemann surface.
A quite general result about asymptotics of Bergman kernel on non-compact manifolds
was given in \cite{MM07,MM08,MM12,HsM11}.
Let $(Y,\omega)$ be a complete K\"ahler manifold of dimension $n$
and $(L,h)\to X$ be a holomorphic Hermitian line bundle such that
\begin{equation}\label{e:c} 
c_1(L,h)\geq \varepsilon\omega\,,\:\:\ric_\omega\geq-C\omega\,,
\end{equation}
for some constants $\varepsilon, C>0$. 
If $(L,h)=(K_Y,h^{K_Y})$, where $K_Y= \det(T^{*(1,0)}Y)$ is the canonical bundle
and $h^{K_Y}$ is induced by $\omega$, 
condition \eqref{e:c} is to be replaced by
\begin{equation}\label{e:c1} 
\ric_\omega\leq-\varepsilon\omega\,.
\end{equation}
Under these assumptions it is known that the Bergman kernel function $P_p(x)$ of 
the space of $L^2$-holomorphic sections $H^0_{(2)}(Y,L^p)$ has the following expansion
\begin{equation}\label{e:bke1} 
\begin{split}
&P_p(x)=b_0(x)p^n+b_1(x)p^{n-1}+\ldots=\sum_{j=0}^\infty b_j(x)p^{n-j}\,,\\
&\text{uniformly on compact sets relative to any $\cC^\ell$-topology.}
\end{split}
\end{equation}
More precisely, there exist
coefficients $b_j\in\cC^\infty(Y)$, $j\in\N$, such that for any compact set $K\subset Y$,
any $k,\ell\in\N$, there exists $C_{k,\ell,K} > 0$ such that for $p\in\N$
\begin{equation}\label{e:bke2}
\left\vert\frac{1}{p^n} P_p(x) -\sum_{j=0}^k b_j(x)p^{-j}\right\vert_{\cC^\ell(K)}
\leq C_{k,\ell,K}\,p^{-k-1}.
\end{equation}
Moreover, we have
\begin{equation}\label{e:coeff}
b_0=\frac{c_1(L,h)^n}{\omega^n}\,,\:\: 
b_1=\frac{1}{8\pi}\,\frac{c_1(L,h)^n}{\omega^n}
\left(r^Y-2\Delta\log\left(\frac{c_1(L,h)^n}{\omega^n}\right)\right)\,,
\end{equation}
where $r^Y$ and $\Delta$ are the scalar curvature and the Bochner Laplacian 
of the metric associated to the K\"ahler form $c_1(L,h)$.

Assume now that $X$ is compact and $c_1(L,h)=\omega$, hence $b_0=1$. 
The expansion $P_p(x)=p^n+O(p^{n-\frac12})$ was proved by Tian \cite[Section\,3]{Ti90} 
in the $\cC^4$-topology and generalized by Ruan \cite{Ru98} to $P_p(x)=p^n+O(p^{n-1})$
in any $\cC^\ell$-topology. Berndtsson \cite{Be03} gave a simple proof of the uniform convergence
$P_p(x)=b_0(x)p^n+o(p^n)$. The asymptotics \eqref{e:bke1} were proved by Catlin \cite{Ca99} and 
Zelditch \cite{Z98}.

In the quantum Hall effect (QHE) the density of states for the lowest 
Landau level on a Riemann surface, or more generally on a K\"ahler manifold, 
is given by the Bergman kernel on the diagonal, 
see \cite{DouKle10} where \eqref{e:bke1} 
was derived using perturbation theory for the quantum mechanical path integral. 
The metric dependence and gravitational anomaly in the quantum Hall states has recently 
been studied using the asymptotic expansion of the Bergman kernel \cite{SK,KMMW} 
as well as other methods \cite{AG1,AG2,CLW1,CLW2,FK14,BR,KW}, see \cite{K16} for a review. 
The quantum Hall states and the density function have been studied recently for surfaces 
with conical singularities \cite{CCLW,Kl16}, singular surfaces with 
$Z_n$-symmetry \cite{Gr16} and for cusps \cite{C16}, see \cite{AMM} 
for the results for the Bergman kernel. Remarkably, 
the quantum Hall effect  on a cone 
can also be realized experimentally, see \cite{S16}, where synthetic 
Landau levels on a cone were constructed in a photon resonator. 
In this paper we study the asymptotic profile of the Bergman kernel 
for more general singular geometries.

If $X$ is non-compact, \eqref{e:bke1} was deduced in 
\cite[Theorem\,6.1.1]{MM07}, \cite[Theorem\,3.11]{MM08} under hypothesis \eqref{e:c}, 
resp.\ \eqref{e:c1}, see also \cite[Theorem\,1.6]{HsM11}. 
We refer the reader to the book \cite{MM07} for a comprehensive study of the Bergman kernel 
and its applications and also to the survey \cite{Ma10}.

\par By the above mentioned works, the asymptotic expansion of the Bergman kernel
is well understood on fixed compact sets. 
In this paper we consider very general metrics
on a punctured Riemann surface and on a line bundle and study the behavior 
of the Bergman kernel in the neighborhood of the punctures. The asymptotics depend on the
singularities of the metrics on the base manifold and on the bundle.

In the case of metrics with Poincar\'e singularities,
\cite{AMM} provides a weighted estimate in the $C^m$-norm
 near the punctures for the global Bergman kernel compared to the Bergman kernel 
of the punctured disc, uniformly in the tensor powers of the given bundle.
Our estimates complement the results of \cite{AMM}.
 
More precisely, we consider in this paper the following setting:

\medskip

(A) $X$ is a Riemann surface and $\Sigma=\{x_j:\,j\geq1\}\subset X$ 
is a discrete closed subset. We fix a smooth Hermitian metric $\Omega$ on $X$ 
and denote by $\dist$ the induced distance.

\smallskip

(B) $\omega$ is a Hermitian metric of class $\cC^2$ on $X\setminus\Sigma$ 
such that $\omega\geq c\,\Omega$, where $c:X\to(0,\infty)$ is a continuous function
and $\ric_\omega\geq-2\pi B\omega$, for some constant $B>0$. 

\smallskip

(C) $L$ is a holomorphic line bundle on $X$ and $h$ is a Hermitian metric of class 
$\cC^3$ on $L\vert_{X\setminus\Sigma}$ such that $c_1(L,h)\geq2\varepsilon\omega$ 
on $X\setminus\Sigma$, for some constant $\varepsilon>0$. 

\medskip

\par Let $h_p:=h^{\otimes p}$ be the metric induced by $h$ on $L^p\vert_{X\setminus\Sigma}$, 
where $L^p:=L^{\otimes p}$. We denote by $H^0_{(2)}(X\setminus\Sigma,L^p)$ 
be the Bergman space of $L^2$-holomorphic sections of $L^p$ relative to the metric 
$h_p$ and $\omega$, 
\begin{equation}\label{e:bs}
H^0_{(2)}(X\setminus\Sigma,L^p)=\left\{S\in H^0(X\setminus\Sigma,L^p):\,
\|S\|_p^2:=\int_{X\setminus\Sigma}|S|^2_{h_p}\,\omega<\infty\right\},
\end{equation}
endowed with the obvious inner product. 

\par Let $d_p\in\N\cup\{\infty\}$ be the dimension of $H^0_{(2)}(X\setminus\Sigma,L^p)$. 
We denote by $P_p$ the Bergman kernel function, 
of the space $H^0_{(2)}(X\setminus\Sigma,L^p)$, which is defined as follows. 
For $p\geq1$, if $\{S_\ell^p\}_{\ell\geq1}$ is an orthonormal 
basis of $H^0_{(2)}(X\setminus\Sigma,L^p)$ , then 
\begin{equation}\label{e:BFS1}
P_p(x)=\sum_{\ell=1}^{d_p}|S^p_\ell(x)|_{h_p}^2\,.
\end{equation} 
Note that $P_p$
is independent of the choice of basis (see \cite[Lemma 3.1]{CM11}). 

\par We fix $x_j\in\Sigma$ and a constant $R_j>0$ with the property that $x_j$ 
has a coordinate neighborhood $U_{x_j}$ centered at $x_j$ such that 
the coordinate disc $B(x_j,2R_j)\Subset U_{x_j}$ and 
\begin{equation}\label{e:Rj}
\dist(B(x_j,2R_j),\Sigma\setminus\{x_j\})\geq R_j.
\end{equation}
Let $e_j$ be a local holomorphic frame of $L$ on $B(x_j,2R_j)$ and 
let $\varphi_j$ be the subharmonic weight of $h$ on $B(x_j,2R_j)\setminus\{x_j\}$ 
corresponding to $e_j$, so $|e_j|_h=e^{-\varphi_j}$. We assume that, in local coordinate 
$z$ on $U_{x_j}$, $\varphi_j$ has the form 
\begin{equation}\label{e:varphi}
\varphi_j(z)=\nu_j\log|z|+\psi_j(z),\,\text{ where }\nu_j\in\R,\;\psi_j\in\cC^3(B(x_j,2R_j)\setminus\{x_j\}).
\end{equation}
Moreover, we assume that there exist constants $A_j>0,\,\alpha_j\geq0$ such that 
the third order derivatives of $\psi_j$ verify
\begin{equation}\label{e:dpsi}
|D^\mu\psi_j(z)|\leq A_j|z|^{-\alpha_j}\,,\,\text{ for all } z\in B(x_j,2R_j)\setminus\{x_j\},
\;|\mu|=3.
\end{equation}
In particular, equation \eqref{e:varphi} allows us to consider 
the special cases when $\psi_j$ is bounded or smooth near $x_j$, 
i.e. the metric $h$ has logarithmic singularities at $\Sigma$. 

Next, we can write on $B(x_j,2R_j)\setminus\{x_j\}$,
\begin{equation}\label{e:rho}
\omega(z)=\frac{i}{2}\,\rho_j(z)\,dz\wedge d\overline z=\rho_j(z)\,dm(z),
\end{equation}
where $dm(z)$ is the Lebesgue measure in the coordinate $z$. 
We assume that there exist constants $A'_j>0,\,\beta_j\geq0$ such 
that the first order derivatives of $\rho_j$ verify 
\begin{equation}\label{e:drho}
|D\rho_j(z)|\leq A'_j|z|^{-\beta_j}\,,\,\text{ for all } z\in B(x_j,2R_j)\setminus\{x_j\}.
\end{equation}
Finally we let 
\begin{equation}\label{e:delta}
\delta_j=\max\{8/3,8\beta_j/3,8\alpha_j\}.
\end{equation}

\par In \cite[Section\,2]{Be03} Berndtsson gave a simple proof for the first order asymptotics 
of the Bergman kernel function $P_p(x)=b_0(x)p^n+o(p^n)$ in the case of powers of an 
ample line bundle on a projective manifold. Adapting his methods to our situation we prove the following
asymptotics near the singular points. They show explicitly how the estimates depend on the distance
to the singular points and on the parameters $\alpha_j$, $\beta_j$ which encode the singularities of the metrics.
\begin{Theorem}\label{T:mt1} Let $(X,\Sigma,\omega,L,h)$ 
verify assumptions (A)-(C), and let $x_j\in\Sigma$. Let $R_j$ 
be defined by \eqref{e:Rj}, $\delta_j$ by \eqref{e:delta}, 
and assume that $h,\,\omega$ satisfy \eqref{e:dpsi}, respectively \eqref{e:drho}. 
There exists a constant $C_j>1$ such that if $x\in B(x_j,R_j)\setminus\{x_j\}$ and 
\begin{equation}\label{e:p}
p>C_j\dist(x,x_j)^{-\delta_j}
\end{equation}
then
\begin{equation}\label{e:Bexp00}
\Big|\frac{P_p(x)}{p}\,\frac{\omega_x}{c_1(L,h)_x}-1\Big|
\leq C_j\Big(p^{-1/8}\dist(x,x_j)^{-\alpha_j}+p^{-3/8}\dist(x,x_j)^{-\beta_j}\Big).
\end{equation}
\end{Theorem}

\par We consider next Bergman kernels for powers of the canonical bundle
of a punctured Riemann surface. In addition to the setting in (A) 
let us consider the following condition:

\smallskip

(B$^\prime$) $\omega$ is a smooth Hermitian metric on $X\setminus\Sigma$ 
such that $\omega\geq c\,\Omega$, where $c:X\to(0,\infty)$ is a continuous function
and $\ric_\omega\leq-\lambda\omega$ on $X\setminus\Sigma$, for some constant $\lambda>0$.

\smallskip

The Hermitian metric $\omega$ induces a Hermitian metric $h^{K_X}$ on 
$K_X|_{X\setminus\Sigma}$\,.
We denote by $h_p$ the metric induced by $h^{K_X}$ on $K_X^p$ and by
$H^0_{(2)}(X\setminus\Sigma,K_X^p)$ the space of holomorphic sections
of $K_X^p|_{X\setminus\Sigma}$ which are $L^2$ with respect to the metrics $h_p$ and volume
form $\omega$, cf.\ \eqref{e:bs}. 

As before, let $z:U_{x_j}\to\C$ be a local holomorphic coordinate, with respect to which
the metric $\omega$ has the form \eqref{e:rho}. 
We define the weight $\varphi_j$ on $B(x_j,2R_j)\setminus\{x_j\}$ of $h^{K_X}$ by
\begin{equation}\label{e:dz}
e^{-2\varphi_j}=\vert dz\vert^2_\omega\,,\:\:\varphi_j=\frac12\Big(\log\rho_j-\log2\Big),
\end{equation}
and we further write $\varphi_j$ as in \eqref{e:varphi}. 
\begin{Theorem}\label{T:mt2} Let $(X,\Sigma, \omega)$ verify 
assumptions (A) and (B$^\prime$). 
Assume that $\omega$ satisfies \eqref{e:dpsi} and \eqref{e:drho}. Let $P_p$ be
the Bergman kernel function of $H^0_{(2)}(X\setminus\Sigma,K_X^p)$.
Then there exists a constant $C_j>1$ such that if $x\in B(x_j,R_j)\setminus\{x_j\}$ and 
\begin{equation}\label{e:p1}
p>C_j\dist(x,x_j)^{-\delta_j}
\end{equation}
then
\begin{equation}\label{e:Bexp001}
\left|\frac{P_p(x)}{p}\left(-\frac{2\pi}{R_{\omega}(x)}\right)-1\right|
\leq C_j\Big(p^{-1/8}\dist(x,x_j)^{-\alpha_j}+p^{-3/8}\dist(x,x_j)^{-\beta_j}\Big),
\end{equation}
where $R_{\omega}$ is the Gauss curvature of $\omega$.
\end{Theorem}
In \cite{AMM} we considered the particular situation of Theorem \ref{T:mt2}
when the metric $\omega$ equals the Poincar\'e metric near the
punctures (hence $\alpha_j=\beta_j=3$, cf.\ Lemma \ref{L:abP})
and obtain estimates which are different in nature as 
those in Theorem \ref{T:mt2}. 

See Corollaries \ref{cor:disc1}, \ref{cor:par}, \ref{cor:par2} for applications
of Theorem \ref{T:mt2}.

\medskip

\par Theorem \ref{T:mt1} has the following interesting corollary which gives a uniform estimate 
on the Bergman kernel $P_p(x)$ in the regime where the distance from $x$ to $\Sigma$ 
decreases slower than some power of $1/p$. 

\begin{Corollary}\label{C:mt1} In the setting of Theorems \ref{T:mt1} or \ref{T:mt2}, 
there exists a constant $C_j>1$ such that if $\eta\in[0,1]$, $p>C_j$, 
and $x\in B(x_j,R_j)$ satisfies $\dist(x,x_j)>(C_j/p)^{\eta/\delta_j}$, then 
$$\Big|\frac{P_p(x)}{p}\,\frac{\omega_x}{c_1(L,h)_x}-1\Big|\leq 2C_j\,p^{-(1-\eta)/8}.$$
\end{Corollary}

\par The paper is organized as follows. In Section \ref{S:prelim} 
we recall some facts about singular Hermitian metrics 
on holomorphic line bundles and 
the solution of the $\overline\partial$ equation.
Section \ref{S:pfmt} 
is devoted to the proof of the main results announced in the Introduction.
In Section \ref{S:appl} we apply these results to interesting metrics 
for which the parameters $\alpha_j$ and $\beta_j$ can be explicitly given (metrics with
logarithmic, cuspidal and conical singularities). 
In Section \ref{S:rugby} we calculate the Bergman kernel
of the Riemann sphere with two conical singularities
and study its asymptotics near the singularities.

\smallskip

\noindent
\textbf{\textit{Acknowledgements.}} We would like to thank T. Can and P. Wiegmann for useful discussions.


\section{Preliminaries}\label{S:prelim} 
\subsection{Singular Hermitian holomorphic line bundles}\label{SS:lb} 
Let $L$ be a holomorphic line bundle on a complex manifold $Y$. 
The notion of singular Hermitian metric $h$ on $L$ is defined as follows 
(see \cite{D90}, \cite[p.\,97]{MM07}): if $e_\ell$ is a holomorphic frame of $L$ 
over an open set $U_\ell\subset Y$ then $|e_\ell|^2_h=e^{-2\varphi_\ell}$, 
where $\varphi_\ell\in L^1_{loc}(U_\ell)$. 
If $g_{\ell k}=e_k/e_\ell\in\mathcal O^*_Y(U_\ell\cap U_k)$ are the transition functions of $L$, 
then $\varphi_\ell=\varphi_k+\log|g_{\ell k}|$. The curvature current $c_1(L,h)$ of $h$
 is the current of bidegree $(1,1)$ on $Y$ defined by $c_1(L,h)=dd^c\varphi_\ell$ on $U_\ell$, where $d^c=\frac{1}{2\pi i}(\partial-\overline\partial)$. 
 If $c_1(L,h)\geq0$ then the weight $\varphi_\ell$ is plurisubharmonic on $U_\ell$. 
 When $Y$ is compact then the space $H^0(Y,L)$ of global holomorphic sections of $L$ is finite dimensional.
 
\par Let now $(X,\Sigma,\omega,L,h)$ be as in Theorem \ref{T:mt1} and $P_p$ be the Bergman kernel functions
of the spaces $H^0_{(2)}(X\setminus\Sigma,L^p)$ defined in \eqref{e:BFS1}. 
Then for all $x\in X\setminus\Sigma$,
\begin{equation}\label{e:Pvar}
P_p(x)=\max\Big\{|S(x)|^2_{h_p}:\,S\in H^0_{(2)}(X\setminus\Sigma,L^p),\;\|S\|_p=1\Big\}.
\end{equation}
Let $(Y,J,\omega)$ be a K\"ahler manifold, where $J$ is the complex structure of $Y$ 
and let $g^{TY}$ be the Riemannian metric associated to 
$\omega$ by $g^{TY}(u,v)=\omega(u,Jv)$ for all $u,v\in T_xY$, $x\in Y$. 
Let $\ric$ be the Ricci curvature of $g^{TY}$. The Ricci form $\ric_\omega$ is defined as the 
$(1,1)$-form associated to $\ric$ by
\begin{equation}\label{e:ric1}
\ric_\omega(u,v)=\ric(Ju,v)\,,\quad\text{for any $u,v\in T_{x}Y$, $x\in Y$}.
\end{equation}
The volume form $\omega^n$ induces a metric $h^{K^*_Y}_\omega$ on $K^*_Y$, 
whose dual metric on $K_Y$ is denoted by $h^{K_Y}_\omega$.
For simplicity we denote by $h_p:=(h^{K_Y}_\omega)^{\otimes p}$ the induced metric on $K_Y^p$.
Since the metric $g^{TY}$ is K\"ahler, we have (see e.\,g.\ \cite[Problem\,1.7]{MM07})
\begin{equation}\label{e:ric2}
\ric_\omega=iR^{K_{Y}^{*}}=-iR^{K_{Y}}=-2\pi c_1(K_Y,h^{K_Y})\,.
\end{equation}
Let us consider now the case of dimension $n=1$. 
The canonical bundle of $Y$ is just $K_Y=T^{(1,0)*}Y$ and
$K^*_Y=T^{(1,0)}Y$, moreover the metric $h^{K^*_Y}_\omega$ on $K^*_Y$ 
is directly given by $\omega$.  In local holomorphic coordinates $z:U\to\C$ we write
$\omega(z)=\frac{i}{2}\,\rho(z)\,dz\wedge d\overline z$, so $g^{TY}(z)=\rho(z)|dz|^2$. 
The Gauss curvature of $g^{TY}$ (and by a slight abuse, of $\omega$) is defined by
\begin{equation}\label{e:gauss}
R_\omega=-\frac{\:2\:}{\:\rho\:}\,\frac{\partial^2}{\partial z\partial\overline{z}}\log\rho.
\end{equation}  
Since the metric $h^{K^*_Y}_\omega$ on $K^*_Y$ 
is directly given by $\omega$, we have that $\partial/\partial z$ is a frame
of $K^*_Y$, $dz$ is the dual frame of $K_Y$ on $U$ and
\begin{equation}\label{e:dz1}
\Big\vert\frac{\partial}{\partial z}\Big\vert^2_\omega=\frac{\rho(z)}{2}\,,\:\:
\vert dz\vert^2_\omega=\frac{2}{\rho(z)}\,\cdot
\end{equation}
The weight $\varphi$ of $h^{K_Y}_\omega$ on $U$  
is given by
\begin{equation}\label{e:dz2}
e^{-2\varphi}=\vert dz\vert^2_\omega\,,\:\:\varphi=\frac12\Big(\log\rho-\log2\Big),
\end{equation}
hence
\begin{equation}\label{e:ric3}
\ric_\omega=-2\pi c_1(K_Y,h^{K_Y})=-2\pi dd^c\varphi=-\pi dd^c\log\rho=R_\omega\,\omega\,.
\end{equation}
In local normal
coordinates associated with $\omega$ near an arbitrary point $x_{0}\in Y$ we have
$\omega|_{x_{0}}= \frac{i}{2}dz\wedge d \overline{z}$
and the scalar curvature $r_\omega$
of $(Y, \omega)$ is given at $x_{0}$ by
\begin{align}\label{eq:mm3.3}
r_\omega= 4 R^{T^{(1,0)}Y}\Big(
\frac{\partial}{\partial z}, \frac{\partial}{\partial \overline{z}}\Big).
\end{align}
Thus
\begin{equation}\label{eq:mm3.2}
-\frac{i}{2} r_\omega \, \omega= R^{T^{(1,0)}Y}
= -R^{K_{Y}}
=\overline{\partial} \partial \log |\sigma|^2,
\end{equation}
where $\sigma$ is a local holomorphic frame of $T^{(1,0)}Y$.
From \eqref{e:ric3} and \eqref{eq:mm3.2} we deduce the relation between
scalar and Gaussian curvature,
\begin{equation}\label{e:ricscal}
r_\omega=2R_\omega\,.
\end{equation}

\subsection{$L^2$-estimates for $\db$}\label{SS:db} 
The following version of Demailly's estimates for the $\db$ operator 
\cite[Th\'eor\`eme 4.1]{D82} will be needed in our proofs (see also \cite[Theorem 2.5]{CMM14}).
\begin{Theorem}[{\cite{D82}}]\label{T:db} Let $Y$ be a complete 
K\"ahler manifold, $\dim Y=n$, and let $\omega$ be a K\"ahler form of class $\cC^2$ 
on $Y$ (not necessarily complete) such that its Ricci form $\ric_\omega\geq-2\pi B\omega$ on $Y$, 
for some constant $B>0$. 
Let $(L,h)$ be a Hermitian holomorphic line bundle on $Y$ such that $h$ is of class $\cC^2$ and
$c_1(L,h)\geq2\varepsilon\omega$. If $p\geq B/\varepsilon$ and $g\in L^2_{0,1}(Y,L^p,loc)$ verifies
\[
\db g=0\:\:\text{and}\:\:\int_Y|g|^2_{h_p}\,\omega^n<\infty
\]
then there exists $u\in L^2_{0,0}(Y,L^p,loc)$ such that
\[
\db u=g\:\:\text{and}\:\:\int_Y|u|^2_{h_p}\,\omega^n\leq\frac{1}{p\varepsilon}\,
\int_Y|g|^2_{h_p}\,\omega^n\,.
\] 
\end{Theorem}
We also need the following version for powers of the canonical bundle.
\begin{Theorem}\label{T:db1} 
Let $Y$ be a complete K\"ahler manifold, $\dim Y=n$, and let $\omega$ be a K\"ahler 
form on $Y$ (not necessarily complete) such that its Ricci form satisfies
$\ric_\omega\leq-\lambda\omega$ on $Y$, for some constant $\lambda>0$. 
If $p\geq2$ and 
$g\in L^2_{0,1}(Y,K_Y^p,loc)$ verifies 
\[
\db g=0\:\:\text{and}\:\:\int_Y|g|^2_{h_p}\,\omega^n<\infty
\] 
then there exists $u\in L^2_{0,0}(Y,K_Y^p,loc)$ such that 
\[
\db u=g\:\:\text{and}\:\:\int_Y|u|^2_{h_p}\,\omega^n\leq\frac{1}{(p-1)\lambda}\,
\int_Y|g|^2_{h_p}\,\omega^n\,.
\] 
\end{Theorem}

\section{Proof of main results}\label{S:pfmt}

In this Section we prove Theorems \ref{T:mt1} and \ref{T:mt2} together
with Corollary \ref{C:mt1}. We then give a semi-global version of these
results (Theorem \ref{T:mt3}).


\begin{proof}[Proof of Theorem \ref{T:mt1}] We use methods from \cite[Section 2]{Be03} 
(see also \cite[Theorem 1.3]{CMM14}), and divide the proof into three steps. Recall the definition \eqref{e:Rj} of $R_j$. 

\medskip

{\em Step 1.} Given $x\in B(x_j,R_j)\setminus\{x_j\}$ we  estimate the growth of the functions 
$\rho_j$ and $\psi_j$ defined in \eqref{e:rho}, resp. \eqref{e:varphi}. 

\par Note that since $\omega\geq c\,\Omega$ we have 
\begin{equation}\label{e:c0}
\rho_j(z)\geq c_0,\;\forall\,z\in B(x_j,2R_j)\setminus\{x_j\},
\end{equation}
for some constant $c_0>0$. Let $x\in B(x_j,R_j)$, $r<|x|/2$, and set 
$$M_j(x,r)=\max\{\rho_j(z):\,|z-x|\leq r\}\,,\;m_j(x,r)=\min\{\rho_j(z):\,|z-x|\leq r\}.$$
Hence 
\begin{equation}\label{e:omega}
m_j(x,r)\,dm(z)\leq\omega(z)\leq M_j(x,r)\,dm(z)\,\text{ on } B(x,r).
\end{equation}
Since $r<|x|/2$ we obtain by \eqref{e:drho}
$$|\rho_j(z)-\rho_j(x)|\leq\frac{A'_jr}{(|x|-r)^{\beta_j}}\leq\frac{2^{\beta_j}A'_jr}{|x|^{\beta_j}}\,,\;z\in B(x,r).$$
Therefore using \eqref{e:c0} we get 
\begin{eqnarray*}
&&M_j(x,r)\leq\rho_j(x)+\frac{2^{\beta_j} A'_jr}{|x|^{\beta_j}}\leq\rho_j(x)
\Big(1+\frac{2^{\beta_j}A'_jr}{c_0|x|^{\beta_j}}\Big),\\
&&m_j(x,r)\geq\rho_j(x)-\frac{2^{\beta_j} A'_jr}{|x|^{\beta_j}}\geq\rho_j(x)
\Big(1-\frac{2^{\beta_j} A'_jr}{c_0|x|^{\beta_j}}\Big).
\end{eqnarray*}
If $r<|x|/4$ and $r<c_0|x|^{\beta_j}/(2^{\beta_j+2}A'_j)$ these estimates yield
\begin{equation}\label{e:M/m}
\frac{M_j(x,2r)}{m_j(x,r)}\leq\frac{1+\dfrac{2^{\beta_j+1}A'_jr}{c_0|x|^{\beta_j}}}
{1-\dfrac{2^\beta_j A'_jr}{c_0|x|^{\beta_j}}}\leq1+\frac{C'_1r}{|x|^{\beta_j}}\,,
\end{equation}
with some constant $C'_1>0$. Note that \eqref{e:M/m} 
holds also with $\dfrac{\rho_j(x)}{m_j(x,r)}$ and $\dfrac{M_j(x,2r)}{\rho_j(x)}$ 
in place of $\dfrac{M_j(x,2r)}{m_j(x,r)}$, since the first two quantities are bounded above by the third. 

\par We next turn our attention to the weight $\varphi_j$ of the metric $h$ corresponding to the local holomorphic frame $e_j$ of $L$ on $B(x_j,2R_j)$ (see \eqref{e:varphi}). Using the Taylor expansion of order 2 of $\psi_j$ at $x$ on $B(x,|x|)$ we can write 
$$\varphi_j(z)=\nu_j\log|z|+\re f_j(z)+\lambda_x|z-x|^2+\widetilde\psi_j(z),$$
where $f_j$ is a holomorphic polynomial, $\widetilde\psi_j$ vanishes to order 3 at $x$, and if $r<|x|/2$ we have by \eqref{e:dpsi} that 
\begin{equation}\label{e:tpsi}
\max\big\{|\widetilde\psi_j(z)|:\,z\in B(x,r)\big\}\leq\frac{A_jr^3}{(|x|-r)^{\alpha_j}}\leq\frac{2^{\alpha_j}A_jr^3}{|x|^{\alpha_j}}\,.
\end{equation}
Since $c_1(L,h)_x\geq2\varepsilon\omega_x$ it follows by \eqref{e:c0} that
\begin{equation}\label{e:lambdax}
\lambda_x\geq\pi\varepsilon\rho_j(x)\geq\pi\varepsilon c_0.
\end{equation}
Note that the function $\log|z|$ is harmonic on the disc $B(x,|x|)$. Hence there exists a holomorphic function $F_j(z)$ on $B(x,|x|)$ such that $\varphi_j(z)=\re F_j(z)+\lambda_x|z-x|^2+\widetilde\psi_j(z)$. Consider the holomorphic frame $e_x=e^{F_j}e_j$ of $L$ on $B(x,|x|)$, so 
\begin{equation}\label{e:tvarphi}
\widetilde\varphi_j(z)=-\log|e_x(z)|_h=\varphi_j(z)-\re F_j(z)=\lambda_x|z-x|^2+\widetilde\psi_j(z)
\end{equation}
is the corresponding weight of $h$. Note that $\widetilde\varphi_j(x)=0$. 

\par We conclude Step 1 by introducing the following function which will be needed in the sequel:
$$E(r):=\int_{|\xi|\leq r}e^{-2|\xi|^2}\,dm(\xi)=\frac{\pi}{2}\,\left(1-e^{-2r^2}\right),$$ 
where $\,dm$ is the Lebesgue measure on $\C$. If $r\geq\delta>0$ then 
\begin{equation}\;\label{e:E(r)}
\frac{\pi}{2E(r)}=1+\frac{e^{-2r^2}}{1-e^{-2r^2}}\leq1+\frac{e^{-2r^2}}{1-e^{-2\delta^2}}\,.
\end{equation}

\medskip

{\em Step 2.} We obtain here the upper estimate for $P_p(x)$ if $x\in B(x_j,R_j)\setminus\{x_j\}$. Let $S\in H^0_{(2)}(X\setminus\Sigma,L^p)$ and write $S=se_x^{\otimes p}$, where $e_x$ is the local holomorphic frame of $L$ on $B(x,|x|)$ from Step 1 and $s\in \mathcal O_X(B(x,|x|))$. Let $r_p\in(0,|x|/4)$ be an arbitrary number which will be specified later. It follows from the sub-averaging inequality for subharmonic functions that 
$$|S(x)|^2_{h_p} = |s(x)|^2 \leq \frac{\int_{B(x,r_p)}|s(z)|^2e^{-2p\lambda_x|z-x|^2}\,dm(z)}{\int_{B(x,r_p)}e^{-2p\lambda_x|z-x|^2}\,dm(z)}\,\cdot$$
Using \eqref{e:omega}, \eqref{e:tvarphi} and \eqref{e:tpsi} we obtain 
\begin{eqnarray*}
\int_{B(x,r_p)}|s(z)|^2e^{-2p\lambda_x|z-x|^2}\,dm(z)&\leq&\frac{\exp\!\big(2p\max_{B(x,r_p)}\widetilde{\psi_j}\big)}{m_j(x,r_p)}\int_{B(x,r_p)}|s(z)|^2e^{-2p\widetilde\varphi_j(z)}\,\omega(z)\\ 
&\leq&\frac{\exp\!\big(2^{\alpha_j+1}A_jpr_p^3|x|^{-\alpha_j}\big)}{m_j(x,r_p)}\,\|S\|^2_p\,.
\end{eqnarray*}
Moreover,
\begin{equation}\label{e:Gauss}
\int_{B(x,r)}e^{-2p\lambda_x|z-x|^2}\,dm(z)=\frac{1}{p\lambda_x}\,E\big(r\sqrt{p\lambda_x}\big)\leq\frac{\pi}{2p\lambda_x}\,,\,\;\forall\,r>0.
\end{equation}
Combining these estimates it follows that
\begin{equation}\label{e:b1}
|S(x)|_{h_p}^2\leq\|S\|_p^2\,\frac{p\lambda_x}{E\big(r_p\sqrt{p\lambda_x}\big)}\,\frac{\exp\!\big(2^{\alpha_j+1}A_jpr_p^3|x|^{-\alpha_j}\big)}{m_j(x,r_p)}\,.
\end{equation}
Note that 
$$\frac{c_1(L,h)_x}{\omega_x}=\frac{2\lambda_x}{\pi\rho_j(x)}\,.$$
Taking the supremum in \eqref{e:b1} over $S\in H^0_{(2)}(X\setminus\Sigma,L^p)$ with $\|S\|_p=1$ we get by \eqref{e:Pvar} 
$$P_p(x)\leq p\,\frac{c_1(L,h)_x}{\omega_x}\,\frac{\pi}{2E\big(r_p\sqrt{p\lambda_x}\big)}\,\frac{\rho_j(x)}{m_j(x,r_p)}\,\exp\!\big(2^{\alpha_j+1}A_jpr_p^3|x|^{-\alpha_j}\big).$$
Using \eqref{e:lambdax} we obtain  
$$P_p(x)\leq p\,\frac{c_1(L,h)_x}{\omega_x}\,\frac{\pi}{2E\big(r_p\sqrt{\pi\varepsilon c_0p}\big)}\,\frac{\rho_j(x)}{m_j(x,r_p)}\,\exp\!\big(2^{\alpha_j+1}A_jpr_p^3|x|^{-\alpha_j}\big).$$
If $r_p\sqrt p\geq1$ then by \eqref{e:E(r)},  
$$\frac{\pi}{2E\big(r_p\sqrt{\pi\varepsilon c_0p}\big)}\leq1+C'_2\exp(-2\pi\varepsilon c_0pr_p^2)$$
with a constant $C'_2>0$. Moreover, if $pr_p^3<|x|^{\alpha_j}$ then 
\begin{equation}\label{e:exp}
\exp\!\big(2^{\alpha_j+1}A_jpr_p^3|x|^{-\alpha_j}\big)\leq\exp\!\big(2^{\alpha_j+4}A_jpr_p^3|x|^{-\alpha_j}\big)\leq1+C'_3pr_p^3|x|^{-\alpha_j},
\end{equation}
with a constant $C'_3>0$. If, in addition, $r_p<c_0|x|^{\beta_j}/(2^{\beta_j+2}A'_j)$ 
then applying \eqref{e:M/m} together with these estimates to the above upper bound on $P_p(x)$ yields 
\begin{equation}\label{e:b2}
\begin{split}
&P_p(x)\leq p\,\frac{c_1(L,h)_x}{\omega_x}\,\Big(1+C'_2e^{-2\pi\varepsilon c_0pr_p^2}\Big)
\Big(1+C'_1r_p|x|^{-\beta_j}\Big)\Big(1+C'_3pr_p^3|x|^{-\alpha_j}\Big),\\
&\text{provided that } 0<r_p<|x|/4,\;r_p\sqrt p\geq1,\;r_p<c_0|x|^{\beta_j}/(2^{\beta_j+2}A'_j),\;pr_p^3<|x|^{\alpha_j}.
\end{split}
\end{equation}
Set $r_p=p^{-a}$, so $r_p\sqrt{p}=p^{1/2-a}$, $pr_p^3=p^{1-3a}$. We have shown the following: 
\begin{equation}\label{e:b3}
\begin{split}
&\text{If } 1/3<a<1/2,\;p^{-a}<|x|/4,\;p^{-a}<c_0|x|^{\beta_j}/(2^{\beta_j+2}A'_j),\;p^{1-3a}<|x|^{\alpha_j}\,,\,\text{ then}\\
&P_p(x)\leq p\,\frac{c_1(L,h)_x}{\omega_x}\,\Big(1+C'_2e^{-2\pi\varepsilon c_0p^{1-2a}}\Big)\Big(1+C'_1p^{-a}|x|^{-\beta_j}\Big)\Big(1+C'_3p^{1-3a}|x|^{-\alpha_j}\Big).
\end{split}
\end{equation}
With $a=3/8$ this implies that there exists a constant $C'_4>0$ such that if 
$$p>C'_4\max\{|x|^{-8/3},|x|^{-8\beta_j/3},|x|^{-8\alpha_j}\}=C'_4|x|^{-\delta_j},$$ 
where $\delta_j$ is defined in \eqref{e:delta}, then 
\begin{equation}\label{e:b4}
\frac{P_p(x)}{p}\,\frac{\omega_x}{c_1(L,h)_x}\leq1+C'_4\Big(p^{-3/8}|x|^{-\beta_j}+p^{-1/8}|x|^{-\alpha_j}\Big).
\end{equation}

\medskip

{\em Step 3.} We obtain now the lower estimate for $P_p(x)$ if $x\in B(x_j,R_j)\setminus\{x_j\}$. As before, let $r_p\in(0,|x|/4)$ be an arbitrary number which will be specified later. Let $\chi:\C\to[0,1]$ be a smooth function such that $\chi=1$ on the unit disc $B(0,1)$ and $\supp\chi\subset B(0,2)$. If $e_x$ is the local holomorphic frame of $L$ on $B(x,|x|)$ from Step 1, define 
$$\chi_p(z)=\rho_j(x)^{-1/2}\,\chi\Big(\frac{z-x}{r_p}\Big)\,,\,\;F=\chi_pe_x^{\otimes p},\,\text{ so }|F(x)|^2_{h_p}=\rho_j(x)^{-1}e^{-2p\widetilde\varphi_j(x)}=\rho_j(x)^{-1}.$$
Using \eqref{e:omega}, \eqref{e:tvarphi}, \eqref{e:tpsi} and \eqref{e:Gauss} we obtain 
\begin{eqnarray}\label{e:b5}
\int_{B(x,2r_p)}e^{-2p\widetilde\varphi_j}\,\omega
&\leq&M_j(x,2r_p)\exp\!\big(2p\max_{B(x,2r_p)}|\widetilde{\psi_j}|\big)\int_{B(x,2r_p)}e^{-2p\lambda_x|z-x|^2}\,dm(z)\nonumber\\ 
&\leq&\frac{\pi M_j(x,2r_p)}{2p\lambda_x}\,\exp\!\big(2^{\alpha_j+4}A_jpr_p^3|x|^{-\alpha_j}\big).
\end{eqnarray}
Since $\chi_p^2\leq\rho_j(x)^{-1}$, this implies 
\begin{equation}\label{e:b6}
\|F\|^2_p=\int_{B(x,2r_p)}\chi_p^2e^{-2p\widetilde\varphi_j}\,\omega\leq\frac{\pi}{2p\lambda_x}\,\frac{M_j(x,2r_p)}{\rho_j(x)}\,\exp\!\big(2^{\alpha_j+4}A_jpr_p^3|x|^{-\alpha_j}\big).
\end{equation}

\par Note that any non-compact Riemann surface admits a complete K\"ahler metric,
since it is a Stein manifold by Behnke-Stein \cite{BS49}. Hence
if $X$ is a Riemann surface and $\Sigma$ is a discrete closed set, then
$X\setminus\Sigma$ admits a complete K\"ahler metric.
By assumptions (B) and (C), $\ric_\omega\geq-2\pi B\omega$, 
$c_1(L^p,h_p)\geq2p\varepsilon\omega$ on $X\setminus\Sigma$. 
So if $p\geq B/\varepsilon$ we can solve the 
$\db$-equation using \cite{D82} (see Theorem \ref{T:db}): 
if $\theta=\db F\in L^2_{0,1}(X\setminus\Sigma,L^p,loc)$ there exists 
$G\in L^2_{0,0}(X\setminus\Sigma,L^p,loc)$ such that $\db G=\theta=\db F$ and 
$$\|G\|^2_p=\int_{X\setminus\Sigma}|G|^2_{h_p}\omega\leq\frac{1}{p\varepsilon}
\int_{X\setminus\Sigma}|\theta|^2_{h_p}\omega.$$
Since $|\db\chi_p|^2\leq\|\db\chi\|^2\rho_j(x)^{-1}r_p^{-2}$, where 
$\|\overline\partial\chi\|$ denotes the maximum of 
$|\overline\partial\chi|$, we get by \eqref{e:b5}, 
$$\int_{X\setminus\Sigma}|\theta|^2_{h_p}\omega=\int_{B(x,2r_p)}|\db\chi_p|^2e^{-2p\widetilde\varphi_j}\,\omega\leq\frac{\pi\|\db\chi\|^2}{2\lambda_xpr_p^2}\,\frac{M_j(x,2r_p)}{\rho_j(x)}\,\exp\!\big(2^{\alpha_j+4}A_jpr_p^3|x|^{-\alpha_j}\big).$$
Thus  
\begin{equation}\label{e:b7}
\|G\|^2_p\leq\frac{1}{p\varepsilon}\,\frac{\pi\|\db\chi\|^2}{2\lambda_x}\,\frac{1}{pr_p^2}\,\frac{M_j(x,2r_p)}{\rho_j(x)}\,\exp\!\big(2^{\alpha_j+4}A_jpr_p^3|x|^{-\alpha_j}\big).
\end{equation}
Since $\overline\partial G=\overline\partial F=0$ on $B(x,r_p)$, $G$ is holomorphic on $B(x,r_p)$. Hence the estimate \eqref{e:b1} applies to $G$ on $B(x,r_p)$ and gives  
$$|G(x)|_{h_p}^2\leq\|G\|_p^2\,\frac{p\lambda_x}{E\big(r_p\sqrt{p\lambda_x}\big)}\,\frac{\exp\!\big(2^{\alpha_j+1}A_jpr_p^3|x|^{-\alpha_j}\big)}{m_j(x,r_p)}.$$
Using \eqref{e:b7} we obtain
$$|G(x)|_{h_p}^2\leq\rho_j(x)^{-1}\,\frac{\pi\|\db\chi\|^2}{2\varepsilon E\big(r_p\sqrt{p\lambda_x}\big)}\,\frac{1}{pr_p^2}\,\frac{M_j(x,2r_p)}{m_j(x,r_p)}\,\exp\!\big(2^{\alpha_j+5}A_jpr_p^3|x|^{-\alpha_j}\big).$$
If $r_p\sqrt p\geq1$ then $E\big(r_p\sqrt{p\lambda_x}\big)\geq E\big(\sqrt{\pi\varepsilon c_0}\big)$ by \eqref{e:lambdax}. So 
\begin{equation}\label{e:b8}
|G(x)|_{h_p}^2\leq\rho_j(x)^{-1}\,\frac{C'_5}{pr_p^2}\,\frac{M_j(x,2r_p)}{m_j(x,r_p)}\,\exp\!\big(2^{\alpha_j+5}A_jpr_p^3|x|^{-\alpha_j}\big),
\end{equation}
with a constant $C'_5>0$.

\par Set $S=F-G\in H^0(X\setminus\Sigma,L^p)$, as $\db S=\db F-\db G=0$. By \eqref{e:b6} and \eqref{e:b7} we have 
$$\|S\|^2_p\leq(\|F\|_p+\|G\|_p)^2
\leq\frac{\pi}{2p\lambda_x}\,\frac{M_j(x,2r_p)}{\rho_j(x)}\,\exp\!\big(2^{\alpha_j+4}A_jpr_p^3|x|^{-\alpha_j}\big)\Big(1+\frac{\|\db\chi\|}{r_p\sqrt{p\varepsilon}}\Big)^2.$$
Moreover, if 
$$Q(x,r_p):=\frac{\sqrt{C'_5}}{r_p\sqrt{p}}\,\left(\frac{M_j(x,2r_p)}{m_j(x,r_p)}\right)^{1/2}\,\exp\!\big(2^{\alpha_j+4}A_jpr_p^3|x|^{-\alpha_j}\big)<1,$$
then using \eqref{e:b8} it follows that 
$$|S(x)|_{h_p}^2\geq(|F(x)|_{h_p}-|G(x)|_{h_p})^2\geq\rho_j(x)^{-1}(1-Q(x,r_p))^2.$$
Therefore 
$$P_p(x)\geq\frac{|S(x)|_{h_p}^2}{\|S\|^2_p}\geq p\,\frac{c_1(L,h)_x}{\omega_x}
\frac{\rho_j(x)}{M_j(x,2r_p)}\,\frac{(1-Q(x,r_p))^2}
{\Big(1+\frac{\|\db\chi\|}{r_p\sqrt{p\varepsilon}}\Big)^2
\exp\!\big(2^{\alpha_j+4}A_jpr_p^3|x|^{-\alpha_j}\big)}\,.$$
If $r_p<c_0|x|^{\beta_j}/(2^{\beta_j+2}A'_j)$ and $pr_p^3<|x|^{\alpha_j}$ then by 
\eqref{e:M/m} and \eqref{e:exp},
$$Q(x,r_p)\leq\frac{\sqrt{C'_5}}{r_p\sqrt{p}}\,\big(1+C'_1r_p|x|^{-\beta_j}\big)^{1/2}\big(1+C'_3pr_p^3|x|^{-\alpha_j}\big)\leq\frac{C'_6}{r_p\sqrt{p}}\,,$$
for some constant $C'_6>1$. We may assume that 
$C'_6\geq\|\overline\partial\chi\|/\sqrt{\varepsilon}$. Applying this estimate on $Q(x,r_p)$ 
together with \eqref{e:M/m} and \eqref{e:exp} to the above lower bound on $P_p(x)$ yields 
\begin{equation}\label{e:b9}
\begin{split}
&\hspace{4mm}P_p(x)\geq p\,\frac{c_1(L,h)_x}{\omega_x}\,
\frac{\Big(1-\frac{C'_6}{r_p\sqrt{p}}\Big)^2}{\big(1+C'_1r_p|x|^{-\beta_j}\big)
\big(1+C'_3pr_p^3|x|^{-\alpha_j}\big)\Big(1+\frac{C'_6}{r_p\sqrt{p}}\Big)^2}\,,\\
&\text{if } p>B/\varepsilon,\;0<r_p<|x|/4,\;r_p\sqrt p>C'_6,\;
r_p<c_0|x|^{\beta_j}/(2^{\beta_j+2}A'_j),\;pr_p^3<|x|^{\alpha_j}.
\end{split}
\end{equation}

\par We again let $r_p=p^{-a}$, so $r_p\sqrt{p}=p^{1/2-a}$, $pr_p^3=p^{1-3a}$. Then \eqref{e:b9} implies the following: 
\begin{equation}\label{e:b10}
\begin{split}
&\text{If } 1/3<a<1/2,\;p>B/\varepsilon,\;p^{1/2-a}>C'_6,\;p^{-a}<|x|/4,\;p^{-a}<c_0|x|^{\beta_j}/(2^{\beta_j+2}A'_j),\;\\ 
& p^{1-3a}<|x|^{\alpha_j}\,,\,\text{ then}\\
& \hspace{4mm} P_p(x)\geq p\,\frac{c_1(L,h)_x}{\omega_x}\,\frac{\Big(1-C'_6p^{a-1/2}\Big)^2}{\big(1+C'_1p^{-a}|x|^{-\beta_j}\big)\big(1+C'_3p^{1-3a}|x|^{-\alpha_j}\big)\Big(1+C'_6p^{a-1/2}\Big)^2}\,.
\end{split}
\end{equation}
Taking $a=3/8$ we conclude that there exists a constant $C'_7>0$ such that if 
$$p>C'_7\max\{|x|^{-8/3},|x|^{-8\beta_j/3},|x|^{-8\alpha_j}\}=C'_7|x|^{-\delta_j}$$ 
then 
\begin{equation}\label{e:b11}
\frac{P_p(x)}{p}\,\frac{\omega_x}{c_1(L,h)_x}\geq1-C'_7\Big(p^{-3/8}|x|^{-\beta_j}+p^{-1/8}|x|^{-\alpha_j}\Big).
\end{equation}
This concludes the proof of Theorem \ref{T:mt1}.
\end{proof}

\smallskip
\begin{proof}[Proof of Theorem \ref{T:mt2}]
Here $L=K_X$. The proof is analogous to the proof of Theorem \ref{T:mt1}, with the only difference
that we apply Theorem \ref{T:db1} instead of Theorem \ref{T:db} in order to solve
the $\db$-equation in Step 3. 
For the leading term of the expansion observe that 
$c_1(K_Y,h^{K_Y})=(-R_\omega/2\pi)\,\omega$
by \eqref{e:ric3}.
\end{proof}
\smallskip

\begin{proof}[Proof of Corollary \ref{C:mt1}] Let $C_j$ be the constant from Theorem \ref{T:mt1}. 
Then $\dist(x,x_j)>(C_j/p)^{\eta/\delta_j}\geq(C_j/p)^{1/\delta_j}$. 
Hence by Theorem \ref{T:mt1}, using that $\alpha_j/\delta_j\leq1/8$ and $\beta_j/\delta_j\leq3/8$, we obtain
\begin{eqnarray*}
\Big|\frac{P_p(x)}{p}\,\frac{\omega_x}{c_1(L,h)_x}-1\Big|
&\leq&C_j\Big(p^{-1/8}\dist(x,x_j)^{-\alpha_j}+p^{-3/8}\dist(x,x_j)^{-\beta_j}\Big)\\
&\leq&C_j\Big(p^{-1/8}(p/C_j)^{\eta\alpha_j/\delta_j}+p^{-3/8}(p/C_j)^{\eta\beta_j/\delta_j}\Big)\\
&\leq&C_j\Big(p^{-(1-\eta)/8}+p^{-3(1-\eta)/8}\Big)\leq2C_j\,p^{-(1-\eta)/8}.
\end{eqnarray*}
\end{proof}

\medskip

\par We give now a semi-global version of Theorem \ref{T:mt1}. Let $K\subset X$ be a compact set and let $\Sigma\cap K=\{x_1,\ldots,x_m\}$. Fix a constant $R_0>0$ with the property that every point $x\in K$ has a coordinate neighborhood $U_x$ centered at $x$ such that the coordinate disc $B(x,2R_0)\Subset U_x$ and 
$$\dist(B(x_j,2R_0),\Sigma\setminus\{x_j\})\geq R_0,\;1\leq j\leq m.$$
Define 
\begin{eqnarray*}
&&\alpha:=\max\{\alpha_j:\,1\leq j\leq m\}\,,\,\;A:=\max\{A_j:\,1\leq j\leq m\},\\
&&\beta:=\max\{\beta_j:\,1\leq j\leq m\}\,,\,\;A':=\max\{A'_j:\,1\leq j\leq m\}.
\end{eqnarray*}
where $\alpha_j,A_j$ and $\beta_j,A'_j$ are as in \eqref{e:dpsi}, 
respectively \eqref{e:drho}. We have the following:

\begin{Theorem}\label{T:mt3} Let $(X,\Sigma,\omega,L,h)$ 
verify assumptions (A)-(C), and let $K\subset X$ be a compact set with 
$\Sigma\cap K=\{x_1,\ldots,x_m\}$. Assume that $h,\,\omega$ satisfy 
\eqref{e:dpsi}, respectively \eqref{e:drho}. 
Then there exists a constant $C=C(K,\Sigma,\omega,L,h)>1$ such that 
if $x\in K\setminus\Sigma$ and $p>C\dist(x,\Sigma)^{-\delta}$, 
where $\delta=\max\{8/3,8\beta/3,8\alpha\}$, then
$$\Big|\frac{P_p(x)}{p}\,\frac{\omega_x}{c_1(L,h)_x}-1\Big|\leq 
C\Big(p^{-1/8}\dist(x,\Sigma)^{-\alpha}+p^{-3/8}\dist(x,\Sigma)^{-\beta}\Big).$$
\end{Theorem}

\begin{proof} Let $$K':=K\setminus\bigcup_{j=1}^mB(x_j,R_j).$$
There exist a positive number $r_0<R_0$ and points $y_j\in K'$, $1\leq j\leq m'$, 
such that $K'\subset\bigcup_{j=1}^{m'}B(y_j,r_0)$ and 
\begin{equation}\label{e:r0}
\dist(B(y_j,2r_0),\Sigma)\geq r_0,\;1\leq j\leq m'.
\end{equation}

\par We have to estimate $P_p(x)$ for $x\in B(y_j,r_0)$. 
Note that $h$ is of class $\cC^3$ and $\omega$ is of class $\cC^2$ 
in a neighborhood of $\overline B(y_j,2r_0)$. As in \eqref{e:rho} we write 
$\omega(z)=\rho_j(z)\,dm(z)$, where $\rho_j\geq c'_0$ 
on $B(y_j,2r_0)$, $1\leq j\leq m'$, with some constant $c'_0>0$. 
If $M_j(x,r)$ and $m_j(x,r)$ are defined as in Step 1 of the proof of 
Theorem \ref{T:mt1}, we have that \eqref{e:M/m} holds for $r<r_0/2$ 
with $\beta_j=0$ and some constant $C'_1>0$. 
Next, we can choose a holomorphic frame $e_x$ of $L$ on 
$B(x,r_0)\subset B(y_j,2r_0)$ for which the corresponding weight 
$\widetilde\varphi_j$ of $h$ verifies \eqref{e:tvarphi} and 
$$\max\{|\widetilde\psi_j(z)|:\,z\in B(x,r)\}\leq C'_8r^3,\,\;r<r_0,$$
with some constant $C'_8>0$. Moreover, $\lambda_x\geq\pi\varepsilon c'_0$. 

\par Proceeding as in Step 2 and Step 3 of the previous proof, 
we show that there exist $p_0\in\N$ and a constant $C'_9>0$, 
such that if $1/3<a<1/2$, $p>p_0$, and $p^{1/2-a}>C'_9$, then 
\begin{equation}\label{e:b12}
\begin{split}
P_p(x)&\leq p\,\frac{c_1(L,h)_x}{\omega_x}\,
\Big(1+C'_9\big(e^{-2\pi\varepsilon c_0'p^{1-2a}}+p^{-a}+p^{1-3a}\big)\Big),\\
P_p(x)&\geq p\,\frac{c_1(L,h)_x}{\omega_x}\,\Big(1-C'_9\big(p^{a-1/2}+
p^{-a}+p^{1-3a}\big)\Big).
\end{split}
\end{equation}
Choosing $a=3/8$ in \eqref{e:b12} we see that there exists $p'_0\in\N$ 
such that if $p>p'_0$ then 
$$\Big|\frac{P_p(x)}{p}\,\frac{\omega_x}{c_1(L,h)_x}-1\Big|\leq 
3C'_9\,p^{-1/8}.$$
Together with Theorem \ref{T:mt1}, this completes the proof of Theorem \ref{T:mt3}.
\end{proof}


\section{Applications}\label{S:appl}
In this section we examine some situations when the parameters $\alpha$ and
$\beta$ can be explicitly calculated. We consider metrics with logarithmic singularities
and hyperbolic metrics with parabolic singularities (cusps) or conical singularities.

\subsection{Metrics with logarithmic singularities}
Let $X$ be a Riemann surface and $\Sigma\subset X$ be a discrete closed subset. 
Let $\omega$ be a Hermitian metric of class $\cC^2$ on $X$. 
Let $(L,h)$ be a holomorphic line bundle on $X$ with singular metric
$h$, see \cite{D90}, \cite[p.\,97]{MM07}.
We assume that $h$ is smooth on $X\setminus\Sigma$ and
has weights with logarithmic singularities at $\Sigma$,
that is, in \eqref{e:varphi} we have $\varphi_j(z)=\nu_j\log|z|+\psi_j(z)$,
with $\nu_j\geq0$ and $\psi_j\in\cC^3(B(x_j,2R_j)$, where
$\varphi_j(z)=-\log|e_j|_h$ is a local weight around $x_j$.

In this situation we have $\alpha=\beta=0$, so from Theorem \ref{T:mt1}
we obtain immediately the following.
\begin{Corollary}\label{cor:disc}
Let $X$ be a Riemann surface and $\Sigma\subset X$ be a discrete closed subset.
Let $\omega$ be a K\"ahler metric of class $\cC^2$ on $X$ 
such that $\ric_\omega\geq-2\pi B\omega$, for some $B>0$.
Let $(L,h)$ be a holomorphic line bundle on $X$, where $h$ has 
weights with logarithmic singularities at $\Sigma$, it is smooth on 
$X\setminus\Sigma$ and $c_1(L,h)\geq2\varepsilon\omega$ holds 
in the sense of currents on $X$, for some $\varepsilon>0$.
Let $P_p(x)$ be the Bergman kernel function of $H^0_{(2)}(X,L^p)$.
Then for any $x_j\in\Sigma$ and any compact set 
$K\subset X$ with $K\cap\Sigma=\{x_j\}$ there exists
$C_j=C_j(K)>0$ such that if $x\in K$ then
\begin{equation}\label{e:Bexp005}
\Big|\frac{P_p(x)}{p}\,\frac{\omega_x}{c_1(L,h)_x}-1\Big|
\leq C_j\, p^{-1/8}\,,\:\:\text{for $p>C_j\dist(x,x_j)^{-8/3}$}.
\end{equation}
\end{Corollary}
Note that on a compact Riemann surface the Ricci curvature is automatically
bounded below. 
Given a compact Riemann surface, a line bundle
$(L,h_0)$ with smooth metric of positive curvature and a finite set $\Sigma\subset X$, we can
always construct a singular Hermitian metric $h$ with the properties
of Corollary \ref{cor:disc}. We can take $h=h_0\exp(-\varepsilon\psi)$, where
$\varepsilon>0$ is small enough and $\psi$ a smooth function on $X\setminus\Sigma$ with
$\psi(z)=\log|z-x_j|$ in a neighborhood of $x_j\in\Sigma$.

Consider a non-compact Riemann surface $X$ endowed with a 
K\"ahler metric of class $\cC^2$ on $X$ 
such $\ric_\omega\geq-2\pi B\omega$, for some $B>0$. One can consider, for example,
a hyperbolic domain in $\mathbb{P}^1$ endowed with the Poincar\'e metric,
or a domain in $\C$ endowed with the Euclidean metric.
Recall that a hyperbolic domain $X\subset{\mathbb P}^1$
is a domain such that ${\mathbb P}^1\setminus X$ contains at least three points.
Since $X$ is Stein, $X$ admits a strictly subharmonic exhaustion function $\varphi$.
Let $\Sigma\subset X$ be a discrete closed subset.
The metric $\exp(-\chi(\varphi)-\psi)$ on the trivial bundle $L=X\times\C$
satisfies the conditions of Corollary \ref{cor:disc} for some convex increasing
function $\chi:\R\to\R$ and for $\psi$ as above.

\begin{Remark}
In terms of the Nadel multiplier ideal sheaves (see e.\,g.\ \cite[Definition\,2.3.1]{MM07}), 
we deal in Corollary \ref{cor:disc} with a singular metric $h$ on $L$ with $c_1(L,h)\geq2\varepsilon\omega$ 
such that the zero variety of the multiplier ideal sheaves $\cI(h)$ of $h$ equals $\Sigma$. 
We have $H^0_{(2)}(X,L^p)=H^0(X,\cO(L^p)\otimes\cI(h^p))$,
the space of global sections of the sheaf $\cO(L^p)\otimes\cI(h^p)$. 
By \cite[Theorem\,1.8]{HsM11}, if $X$ is compact, the Bergman kernel function 
$P_p(x)$ of  $H^0(X,\cO(L^p)\otimes\cI(h^p))$ has the full asymptotic expansion \eqref{e:bke1} 
on compact sets of $X\setminus\Sigma$.
\end{Remark}
\subsection{Poincar\'e metric on the punctured disc} 
Let us consider $X=\D$ and $\Sigma=\{0\}$. We endow the punctured disc 
$Y:=\D^*=\D\setminus\{0\}$ with the Poincar\'e metric 
$ds^2=(|z|\log|z|^2)^{-2}|dz|^2$, that is,
\begin{equation}\label{e:pm}
\omega=\frac{i}2\frac{dz\wedge d\overline{z}}{|z|^2(\log|z|^2)^2}\,\cdot
\end{equation}
This is a complete K\"ahler metric with Gauss curvature $R_\omega=-4$, or equivalently,
$\ric_\omega=-4\omega$ (see \eqref{e:ric3}). The metric $\omega$ fulfills condition (B$^\prime$).
\begin{Lemma}\label{L:abP}
We have $\alpha=\beta=3$.
\end{Lemma}
\begin{proof}
We use \eqref{e:dz}, \eqref{e:varphi}, 
\eqref{e:dpsi} and \eqref{e:rho}, \eqref{e:drho}.
We have $\varphi(z)=-\log|z|-\log|\log|z|^2|-\frac12\log2$, hence 
$\psi(z)=-\log|\log(x^2+y^2)|-\frac12\log2$, with $z=x+iy$. 
By explicitly calculating $\partial^3_x\psi$
and $\partial^2_x\partial_y\psi$ we obtain that for any $r\in(0,1)$ 
there exists a constant $C=C_r$ such that $|\partial^3_x\psi(z)|\leq C|z|^{-3}$,
$|\partial^2_x\partial_y\psi(z)|\leq C|z|^{-3}$ for $|z|\leq r$. By symmetry we obtain 
the same estimates for $\partial^3_y\psi$, $\partial_x\partial^2_y\psi$. Thus $\alpha=3$.

We have $\rho(z)=(|z|\log|z|^2)^{-2}$. By direct calculation we obtain 
$|\partial_z\rho(z)|\leq|z|^{-3}$ and by symmetry the same estimate for 
$|\partial_{\overline{z}}\rho(z)|$. Thus $\beta=3$.
\end{proof}
By Theorem \ref{T:mt2} and Lemma \ref{L:abP} we obtain:
\begin{Corollary}\label{cor:disc1}
Let $P_p(x)$ be the Bergman kernel function of $H^0_{(2)}(\D^*,K_{\D^*}^p)$.
Let $r\in(0,1)$.
Then there exists $C=C(r)>1$ such that if $0<|x|\leq r$ and $p>C|x|^{-24}$ we have
\begin{equation}\label{e:Bexp002}
\left|\frac{P_p(x)}{p}\cdot\frac{\pi}{2}-1\right|
\leq C p^{-1/8}|x|^{-3}.
\end{equation}
\end{Corollary}
An explicit expression of the Bergman
kernel on the punctured unit disc was derived in \cite[(3.7)]{AMM}. 
Modifying slightly \cite[Proposition 3.3]{AMM} we can show that 
\begin{equation}\label{e:Bexp003}
P_p(x)=\frac2\pi p-\frac4\pi+O(e^{-cp})\,,\:\:p\to\infty,
\end{equation}
outside a fixed neighborhood of the origin
(or more generally, outside a shrinking neighborhood \cite[(3.9)]{AMM}). 
Note that in \cite{AMM} the metric \eqref{e:pm}
is normalized such that its Gauss curvature equals $-2$ and one works
with a line
bundle $L$ satisfying $2\pi c_1(L,h)=\omega$, hence the first two
coefficients of the expansion differ from those in \cite{AMM}
(see also Remark \ref{R:ga}).
Corollary \ref{cor:disc1} is concerned with the behavior
of the Bergman kernel inside a neighborhood of the origin.
\subsection{Hyperbolic metrics with parabolic singularities}
We consider the Riemann sphere ${\mathbb P}^1={\mathbb C}\cup{\infty}$ with $m\geq3$ 
marked points $\Sigma=\{x_1,\ldots,x_m\}$. By using a M\"obius map we can assume 
that $x_m=\infty$. 
Let $Y={\mathbb P}^1\setminus\Sigma$.
By the uniformization theorem,
$Y$ is the quotient $\mathbb{H}/\Gamma$, where $\mathbb{H}$ is the upper half-plane, 
and $\Gamma\subset\operatorname{PSL}(2,\R)$
is a finitely generated torsion-free Fuchsian group acting on $\mathbb{H}$ by linear fractional
transformations. The Poincar\'e metric on $\mathbb{H}$ descends to $Y$ and
gives a complete K\"ahler metric $ds^2$ of constant curvature $-1$.
By \cite[Lemma 2]{TZ88} the metric $ds^2$ has the form $ds^2=e^\varphi|dz|^2$, where $\varphi$ 
is a smooth function on $Y$ verifying $\varphi_{z\overline z}=\frac{1}{2}\,e^\varphi$ and

\begin{equation}\label{e:cusp}
\varphi(z)=\begin{cases}
-2\log|z-x_j|-2\log\big|\!\log|z-x_j|\big|+O(1), & \text{as } z\to x_j,\;1\leq j\leq m-1,\\ 
-2\log|z|-2\log\big|\!\log|z|\big|+O(1), & \text{as }z\to x_m=\infty.
\end{cases}
\end{equation}
\begin{Lemma}\label{L:abP1}
For each $j=1,\ldots,m$, we have  $\alpha_j=\beta_j=3$.
\end{Lemma}
\begin{proof}
Since the curvature is constant $-1$ we have 
$\varphi_{z\overline{z}}=\frac12e^{\varphi}=\rho$. 
Without loss of generality we may assume $j=1$ and $x_1=0$. 
Near $z_1=0$,
we have \[\varphi(z)=-2\log|z|-2\log\big|\!\log|z|\big|+O(1).\] 
By \cite[Lemma 2]{TZ88}, there exist $C>0$ such that near $0$ we have
\begin{equation}\label{e:phiz}
|\varphi_z|\leq \frac{C}{|z|}\,\cdot
\end{equation}
In the rest of the proof we will denote by $C$ a constant may change from line to line.
By \cite[Lemma 2]{TZ88}, see also \cite[(1.7),\,(1.8)]{TZ88}, there exists 
$c_1,\ldots,c_{m-1}$, such that 
\begin{equation}\label{e:phizz}
\varphi_{zz}-\frac12\varphi_z^2=
\sum_{i=1}^{m-1}\left(\frac{1}{2(z-x_i)^2}+\frac{c_i}{z-x_i}\right)\,,
\end{equation}
(the right hand side is actually the Schwarzian derivative of the inverse of the projection 
$\mathbb{H}\to\mathbb{H}/\Gamma=Y$ and $c_1,\ldots,c_{m-1}$ 
are called accessory parameters). By \eqref{e:phiz}, \eqref{e:phizz}, 
\[
|\varphi_{zz}|\leq \frac{C}{|z|^2}\:\:\text{near $z_1=0$},
\]
and 
\[
\varphi_{zzz}=\varphi_z\,\varphi_{zz}-
\sum_{i=1}^{m-1}\left(\frac{1}{(z-z_i)^3}+\frac{c_i}{(z-z_i)^2}\right)
\]
Hence
\[|\varphi_{zzz}|\leq\frac{C}{|z|^3}\:\:\text{near $z_1=0$}.\]
Now $\rho_z=\varphi_{z\,z\,\overline{z}}=\frac12\varphi\,\varphi_z$, 
thus
\[
|\rho_z|=|\varphi_{z\,z\,\overline{z}}|
\leq\frac{C}{|z|^3}\:\:\text{near $z_1=0$}.\]
Further, $\psi=\varphi+2\log|z|$, hence $\psi_z=\varphi_z+z^{-1}$
and $\psi_{zzz}=\varphi_{zzz}+2z^{-3}$, 
$\psi_{z\,z\,\overline{z}}=\varphi_{z\,z\,\overline{z}}$.
We deduce that $\alpha_1=\beta_1=3$.
\end{proof}
By Theorem \ref{T:mt2} and Lemma \ref{L:abP1} we obtain:
\begin{Corollary}\label{cor:par}
Let $Y:={\mathbb P}^1\setminus\Sigma$ with $\Sigma=\{x_1,\ldots,x_m\}$, $m\geq3$, 
be endowed with the induced Poincar\'e metric of constant curvature $-1$.
Let $P_p(x)$ be the Bergman kernel function of $H^0_{(2)}(Y,K_{Y}^p)$.
 Then for any $j=1,\ldots,m$, and any compact set 
$K\subset{\mathbb P}^1$ with $K\cap\Sigma=\{x_j\}$ there exists
$C_j=C_j(K)>0$ such that if $x\in K$ then
\begin{equation}
\left|\frac{2\pi P_p(x)}{p}-1\right|
\leq C_j\, p^{-1/8}\dist(x,x_j)^{-3}\,,\:\:\text{for $p>C_j\dist(x,x_j)^{-24}$}.
\end{equation}
\end{Corollary}
Note that the space $H^0_{(2)}(Y,K_{Y}^p)$ is the space of cusp form 
of weight $2p$ on $Y$, so $P_p(x)$ is the Bergman kernel of the cusp forms.

Let us now point out the interpretation of Corollaries \ref{cor:disc1} and \ref{cor:par}
in terms of classical Bergman kernels for function spaces. Let us consider a general
hyperbolic domain $Y\subset{\mathbb C}$,
i.\,e.\ ${\mathbb C}\setminus Y$ contains at least two points.
As above, $Y$ admits an induced Poincar\'e metric $ds^2=\rho(z)|dz|^2$ of constant curvature $-1$.
For $p\in\N$ we define the Petersson scalar product
\begin{equation}\label{e:pet}
\langle f,g\rangle_p:=\int_D f(z)\overline{g(z)}\rho^{1-p}(z)\,dm(z)\,,
\end{equation}
on the space
$L^2_p(Y):=\big\{\text{$f:Y\to\C$ measurable}: \|f\|_p^2=\langle f,f\rangle_p<\infty\big\}$,
see e.\,g.\ \cite[p.\,88]{Kra}, where $dm(z)=\frac{i}{2}dz\wedge d\overline{z}$ 
is the Euclidean volume form. Set
\begin{equation*}
\mathcal{A}^2_p(Y):=\big\{f\in L^2_p(Y) : \text{$f$ holomorphic}\big\}\,,
\end{equation*}
which is a closed subspace of $L^2_p(Y)$. Denote by $\Pi_p:L^2_p(Y)\to\mathcal{A}^2_p(Y)$
the orthogonal projection and by $\Pi_p(\cdot,\cdot)$ its reproducing kernel. 
If $\{f_j\}$ is an orthonormal basis of 
$(\mathcal{A}^2_p(Y),\langle\cdot,\cdot\rangle)$
then $\Pi_p(z,w)=\sum_jf_j(z)\overline{f_j(w)}$. 
The restriction on the diagonal $z\mapsto\Pi_p(z):=\sum_j|f_j(z)|^2$
is the Bergman
kernel function for the Petersson scalar product.
Note that any element $S\in H^0_{(2)}(Y,K_{Y}^p)$ is of the form
$S=f (dz)^{\otimes p}$, where $f\in\mathcal{O}(Y)$. 
Let $S'=f' (dz)^{\otimes p}$ be a further element of $H^0_{(2)}(Y,K_{Y}^p)$.
By \eqref{e:dz1},
\[\langle S,S'\rangle_\omega\,\omega=f\,\overline{f'}\,|(dz)^{\otimes p}|^2_\omega\,\omega
=2^pf\,\overline{f'}\,\rho^{1-p}dm(z),\]
hence $\langle S,S'\rangle=2^p\langle f,f'\rangle_p$. The map 
$\mathcal{A}^2_p(Y)\to H^0_{(2)}(Y,K_{Y}^p)$, $f\mapsto 2^{-p/2}f(dz)^{\otimes p}$
is an isometry. Therefore
\begin{equation}\label{e:Bpet}
P_p(z)=\Pi_p(z)\rho(z)^{-p}\,,\:\:z\in Y.
\end{equation}
By Corollary \ref{cor:par} and \eqref{e:Bpet} we obtain:
\begin{Corollary}\label{cor:par2}
Let $Y={\mathbb P}^1\setminus\Sigma$ with 
$\Sigma=\{x_1,\ldots,x_m\}$, $m\geq3$, and
write $ds^2=\rho(z)|dz|^2$
for the induced Poincar\'e metric  of constant curvature $-1$ on $Y$.
Let $\Pi_p(z)$ be the Bergman kernel function associated to the 
Petersson scalar product \eqref{e:pet}.
 Then for any $j=1,\ldots,m$, and any compact set 
$K\subset{\mathbb P}^1$ with $K\cap\Sigma=\{x_j\}$ there exists
$C_j=C_j(K)>0$ such that if $z\in K$ then
\begin{equation}\label{e:Bexp004}
\left|\frac{2\pi \Pi_p(z)\rho(z)^{-p}}{p}-1\right|
\leq C_j\, p^{-1/8}\dist(z,x_j)^{-3}\,,\:\:\text{for $p>C_j\dist(z,x_j)^{-24}$}.
\end{equation}
\end{Corollary}
A similar statement can proved for the Bergman kernel $\Pi_p(z)$
of the Petersson scalar product on $\D^*$, by using Corollary \ref{cor:disc1}. 
For the explicit formula of the Bergman kernel $\Pi_p(z)$ for $\D^*$ see
\cite[\S5]{EP96}.
\begin{Remark}\label{R:ga}
Let $Y\subset\mathbb{P}^1$ be an arbitrary hyperbolic domain
endowed with the induced Poincar\'e metric of constant Gauss curvature $-1$.
Let $P_p(x)$ be the Bergman kernel function of $H^0_{(2)}(Y,K_{Y}^p)$. 
By
\cite[Theorem\,6.1.1]{MM07}, \cite[Theorem\,3.11]{MM08}, $P_p(x)$ has a full asymptotic
expansion \eqref{e:bke1} on compact sets of $Y$. 
Actually, by \cite[Corollary 2.4]{AMM} the
expansion reads
\begin{equation}
P_p(x)=\frac{p}{2\pi}-\frac{1}{4\pi}+O(p^{-\infty})\,,\:\:p\to\infty.
\end{equation}
Indeed, $b_0=\frac{1}{2\pi}$ by \eqref{e:ric3}. The Gauss curvature
of $c_1(K_{Y},h^{K_Y})$ being $-2\pi$, the scalar curvature of
$c_1(K_{Y},h^{K_Y})$ is $r^Y=-4\pi$ by \eqref{e:ricscal}. Hence $b_1=-\frac{1}{4\pi}$\,.
All other coefficients vanish by \cite[Corollary 2.4]{AMM}, so the remainder is $O(p^{-\infty})$.

If $Y={\mathbb P}^1\setminus\Sigma$ with $\Sigma=\{x_1,\ldots,x_m\}$, $m\geq3$, 
we refer to \cite{AMM} for a weighted estimate near the punctures
for the global Bergman kernel compared to the Bergman kernel of the punctured
disc.
\end{Remark}
\subsection{Hyperbolic metrics with conical singularities}
We consider again the Riemann sphere ${\mathbb P}^1={\mathbb C}\cup{\infty}$ with $m\geq3$ 
marked points $\Sigma=\{x_1,\ldots,x_m\}$. 
By using a M\"obius map we can assume 
that 
$x_m=\infty$. Suppose that 
$a_j$, $1\leq j\leq m$, are real numbers such that 
\begin{equation}\label{e:aj}
a_j<1\,,\:\:\:\sum_{j=1}^ma_j>2.
\end{equation}
Then 
$$Y:={\mathbb P}^1\setminus\{x_1,\ldots,x_m\}={\mathbb C}\setminus\{x_1,\ldots,x_{m-1}\}$$ 
admits a unique K\"ahler metric of constant curvature $-1$, which on 
${\mathbb P}^1$ is with conical singularities of order $a_j$ (or angle $2\pi(1-a_j))$ 
at $x_j$, $1\leq j\leq m$ (see \cite[Section 2]{TZ03} and references therein). 
This metric has the form $ds^2=e^\varphi|dz|^2$, where $\varphi$ 
is a smooth function on $Y$ verifying 
\begin{equation}\label{e:con}
\varphi_{z\overline z}=\frac{1}{2}\,e^\varphi\,,\,\;\;\varphi(z)=\begin{cases}
-2a_j\log|z-x_j|+O(1), & \text{as } z\to x_j,\;1\leq j\leq m-1,\\ 
-2(2-a_m)\log|z|+O(1), & \text{as }z\to x_m=\infty.
\end{cases}
\end{equation}
\begin{Lemma}
For each $j=1,\ldots,m$, we have  $\alpha_j=\beta_j=1+2a_j$.
\end{Lemma}
\begin{proof}
Since the curvature is constant $-1$ we have 
$\varphi_{z\overline{z}}=\frac12e^{\varphi}=\rho$. 
Without loss of generality we may assume $j=1$ and $x_1=0$. 
By \cite[Lemma\,2 and (9)]{TZ03},  we have
\begin{equation}\label{e:phi}
e^\varphi=\frac{4|w'|}{(1-|w|^2)^2}\,,\:\:\text{with $w(z)=z^{1-a_1}g(z)$,
$g$ holomorphic near $x_1=0$, $g(0)\neq0$.}
\end{equation}
The function $w$ is a multi-valued meromorphic function on $\mathbb{P}^1$
with ramification points at $\{x_1,\ldots,x_m\}$, and it becomes single-valued on the universal 
cover of $Y$.

Let us consider the function $F(z)=(1-a_1)g(z)+zg'(z)$,
holomorphic near $0$, $F(0)=(1-a_1)g(0)\neq0$.
By \eqref{e:phi}, 
\[
\varphi(z)=-2a_1\log|z|+\psi(z)+\log4\,,\:\:\text{where
$\psi(z)=\log|F(z)|^2-2\log(1-|w|^2)$}.
\]
Let us denote
$\widetilde{\psi}(z)=-\log(1-|w|^2)$. We have $\psi_z=(F'/F)+2\widetilde\psi_z$,
with $(F'/F)$ holomorphic near $0$. We have thus to estimate only the derivatives of $\widetilde\psi$.
By a direct computation we obtain
\[
\widetilde\psi_{zzz}=\frac{2(w'\overline{w})^3}{(1-|w|^2)^3}
+\frac{3w'w''\overline{w}^2}{(1-|w|^2)^2}+\frac{w'''\overline{w}}{(1-|w|^2)}\,\cdot
\]
Taking into account \eqref{e:phi}, there exists $C>0$ such that near $0$ we have
\begin{equation}\label{e:w}
|w|\leq C|z|^{1-a_1}\,,\:\:|w'|\leq C|z|^{-a_1}\,,\:\:|w''|\leq C|z|^{-1-a_1}
\,,\:\:|w'''|\leq C|z|^{-2-a_1}\,.
\end{equation}
By \eqref{e:phi} and \eqref{e:w} we infer that near $0$
\[
|\widetilde\psi_{zzz}|\leq C(|z|^{3-6a_1}+|z|^{1-4a_1}+|z|^{-1-2a_1})\leq C|z|^{-1-2a_1}.
\]
We have used here that $a_1<1$. The previous estimate implies
\begin{equation}\label{e:psizzz}
|\psi_{zzz}|\leq C|z|^{-1-2a_1}.
\end{equation}
We estimate now $\rho_z$ and $\psi_{zz\overline{z}}$ simultaneously.
We have $\varphi_{z\overline{z}}=\psi_{z\overline{z}}=
2\widetilde\psi_{z\overline{z}}=\frac12 e^{\varphi}=\rho$
so $\rho_z=\psi_{zz\overline{z}}=\frac12 e^{\varphi}\varphi_z$.
On the other hand, $\varphi_z=-\frac{a_1}{z}+\frac{F'}{F}+2\widetilde\psi_z$,
where $\widetilde\psi_z=\frac{w'\overline{w}}{1-|w|^2}$, so $|\widetilde\psi_z|\leq C|z|^{1-2a_1}$.
We deduce that $|\varphi_z|\leq C|z|^{-1}$. Since $e^\varphi\leq C|z|^{-2a_1}$ we deduce
\begin{equation}\label{e:roz}
|\rho_z|=|\psi_{zz\overline{z}}|\leq C |z|^{-1-2a_1}\,.
\end{equation}
By symmetry, we obtain similar estimates as \eqref{e:psizzz}, \eqref{e:roz} 
for $\psi_{\overline{z}\,\overline{z}\,\overline{z}}$ and $\rho_{\overline{z}}$, 
$\psi_{z\,\overline{z}\,\overline{z}}$, which show that $\alpha_1=\beta_1=1+2a_1$.
\end{proof}
The associated $(1,1)$-form to $ds^2$ is $\omega=\frac{i}{2}\,e^\varphi dz\wedge d\overline z$, thus 
$$\ric_\omega=-i\,\partial\db\log\frac{e^\varphi}{2}
=-i\,\varphi_{z\overline z}\,dz\wedge d\overline z=-\omega.$$ 
Note that $\omega$ induces a Hermitian metric $h_\omega$ on 
$K_Y=K_{{\mathbb P}^1}\vert_Y=\cO_{{\mathbb P}^1}(-2)\vert_Y$ with curvature 
$$c_1(K_Y,h_\omega)=-\frac{1}{2\pi}\,\ric_\omega=\frac{1}{2\pi}\,\omega.$$
\par In the chart $\mathbb C$ the metric $h_\omega$ has weight 
$\varphi_0=\frac{1}{2}\,\log\frac{e^\varphi}{2}=\frac{\varphi}{2}-\frac{\log 2}{2}$. 
Letting $z=1/\zeta$ we obtain in coordinate $\zeta$ near $x_m=\infty$ that 
$$ds^2=e^{\varphi(1/\zeta)}\,\frac{|d\zeta|^2}{|\zeta|^4}
=e^{\varphi'(\zeta)}|d\zeta|^2,\;\varphi'(\zeta)=\varphi(1/\zeta)-4\log|\zeta|
=-2a_m\log|\zeta|+O(1)\text{ as }\zeta\to0.$$
As before, in this chart the weight of 
$h_\omega$ is $\varphi_1=\frac{\varphi'}{2}-\frac{\log 2}{2}$. 
It follows that $h_\omega$ extends to a singular Hermitian metric on $K_{{\mathbb P}^1}$ 
which does not have positive curvature measure since $a_j$ cannot be all $\leq0$.

\par Note that $\omega\geq c\,\Omega$ for some positive metric $\Omega$ on 
${\mathbb P}^1$ if and only if $a_j\geq0$ for all $1\leq j\leq m$. 
We conclude that $({\mathbb P}^1,\Sigma,\omega,K_{{\mathbb P}^1},h_\omega)$ 
verify assumptions (A)-(C) if and only if $a_j\geq0$, $1\leq j\leq m$, 
hence Theorem \ref{T:mt1} applies in this case.

\begin{Corollary}\label{cor:con}
Let $Y:={\mathbb P}^1\setminus\Sigma$ with $\Sigma=\{x_1,\ldots,x_m\}$, $m\geq3$, 
be endowed with the K\"ahler metric $\omega$ of constant curvature $-1$ and conical 
singularities of order $a_j\in[0,1)$ at $x_j$, cf.\ \eqref{e:aj}, \eqref{e:con}.
Let $P_p$ be the Bergman kernel function of $H^0_{(2)}(Y,K_Y^p)$
associated to $\omega$ and $h_\omega$.
 Then for any $j=1,\ldots,m,$ and any compact set 
$K\subset{\mathbb P}^1$ with $K\cap\Sigma=\{x_j\}$ there exists
$C_j=C_j(K)>0$ such that, for $x\in K$ and $p>C_j\dist(x,x_j)^{-8(1+2a_j)}$,
\begin{equation}
\left|\frac{2\pi P_p(x)}{p}-1\right|
\leq C_j\, p^{-1/8}\dist(x,x_j)^{-(1+2a_j)}\,.
\end{equation}
\end{Corollary}


\section{Riemann sphere with two conical singularities}\label{S:rugby}

In this section we calculate explicitly the Bergman kernel and study its scaling asymptotics near conical singularities using rescaled coordinates involving the magnetic length, suggested by \cite{CCLW}. We also interpret our results in terms of the density of states on the lowest Landau level.

\subsection{Metrics with conical singularities}\label{SS:rugby} 
We take the line bundle $L={\mathcal O}(1)$ on the projective space 
${\mathbb P}^1$, endowed with the Hermitian metric $h_a$ given by 
the logarithmically homogeneous plurisubharmonic function on ${\mathbb C}^2$,
\begin{equation}\label{e:phia}
\varphi_a(t,z)=\frac{1}{2a}\,\log(|t|^{2a}+|z|^{2a})\,,\,\;0<a\leq1.
\end{equation}
Consider the standard embedding $z\in{\mathbb C}\hookrightarrow[1:z]\in{\mathbb P}^1$. 
Then $\omega_a:=c_1(L,h_a)$ is given by
\begin{equation}\label{e:oma}
\omega_a|_{_{\mathbb C}}=dd^c\varphi_a(1,z)=
i\,\frac{a}{2\pi |z|^{2(1-a)}(1+|z|^{2a})^2}\,dz\wedge d\overline z\,,
\end{equation}
and is the K\"ahler form associated to the metric 
$$ds^2=\frac{a}{\pi |z|^{2(1-a)}(1+|z|^{2a})^2}\,|dz|^2$$ 
on ${\mathbb P}^1$ with conical singularities of order $a$ 
(or angle $2\pi(1-a))$ at $0$ and $\infty$. 
This surface is sometimes also called american football or spindle. 
Thus $\omega_a$ is polarized by $(L,h_a)$. Moreover, $\int_{{\mathbb P}^1}\omega_a=1$ and 
$$\ric_{\omega_a}
=-i\,\partial\overline\partial\,\log\frac{a}{2\pi |z|^{2(1-a)}(1+|z|^{2a})^2}
=4\pi a\omega_a+2\pi(1-a)(\delta(0)+\delta(\infty))$$
in the sense of currents on ${\mathbb P}^1$, 
where $\delta(0)=\frac{i}{2}\,\delta_0\,dz\wedge d\overline z$ and $\delta_0$ 
is the Dirac measure at $0$. So $\omega_a$ has constant Gauss curvature on 
${\mathbb C}^\star$ (and also on ${\mathbb P}^1$, in the sense of distributions),
$$R_{\omega_a}=\frac{\ric_{\omega_a}}{\omega_a}=4\pi a\,.$$
Let us define the function $\psi$ on ${\mathbb P}^1$,
\begin{equation}\label{e:psi}
\psi([t:z])=\frac{\nu}{2a}\,\log\frac{|z|^{2a}}{|t|^{2a}+|z|^{2a}}\,,\,\text{ where }\nu\in\mathbb R\,,
\end{equation}
and consider on $L^p={\mathcal O}(p)$ the singular Hermitian metric 
\begin{equation}\label{e:hp}
h_p=e^{-2\psi}\,h_a^{\otimes p}\,.
\end{equation}
We have that 
\begin{equation}\label{e:c2}
c_1(L^p,h_p)=(p-\nu)\omega_a+\nu\delta(0)\,.
\end{equation}
The motivation for adding the weight $\psi$ in the metric $h_p$
is to create a $\delta$-distribution independent of $p$ in the curvature, which can
be interpreted as the Aharonov-Bohm flux in the magnetic field, 
see Section \ref{SS:dsLLL}.

Let $P^{a,\nu}_p$ be the Bergman kernel of the Hilbert space 
$H^0_{(2)}({\mathbb P}^1\setminus\{0,\infty\},L^p)$ 
of $L^2$-integrable holomorphic sections of $L^p$ 
relative to the metrics $h_p$ and $\omega_a$. Let $(\cdot,\cdot)_p$ 
denote the corresponding inner product. 

\begin{Proposition}\label{P:rugby}
In the above setting, we have that 
\begin{equation}\label{e:Pp}
P_p(z):=P^{a,\nu}_p(z)=\frac{|z|^{2(j_0-\nu)}}{(1+|z|^{2a})^{\frac{p-\nu}{a}}}\,
\sum_{j=0}^{p-j_0}\frac{|z|^{2j}}{B\left(1+\frac{j+j_0-\nu}{a}\,,1
+\frac{p-j-j_0}{a}\right)}\,,\,\;z\in\mathbb C\,,
\end{equation}
where $B$ is the Euler Beta function,  
$$j_0=\max\{\floor{\nu-a}+1,0\}=\left\{\begin{array}{ll}0,
\mbox{ \hspace{17.1mm} if $\nu<a$},\\ \floor{\nu-a}+1,
\mbox{ if  $\nu\geq a$}, \end{array}\right.$$
and $\floor{x}$ is the largest integer $\leq x$. 
Moreover, $P_p(z)=O\big(|z|^{2(j_0-\nu)}\big)$ for $z$ near 0, and 
$$j_0-\nu=\left\{\begin{array}{ll}-\nu,\mbox{ \hspace{7mm} if $\nu<0$},\\
-\{\nu\},\mbox{ \hspace{2.7mm} if $\nu\geq0$ and $0\leq\{\nu\}<a$},\\
1-\{\nu\},\mbox{ if $\nu\geq0$ and $a\leq\{\nu\}<1$}, 
\end{array}\right.$$
where $\{x\}:=x-\floor{x}$. In particular $0<P_p(0)<+\infty$ 
if and only if $\nu\in\mathbb N$, in which case $P_p(0)=\frac{p-\nu}{a}+1$.
\end{Proposition}

\begin{proof} For $0\leq j\leq p$ we obtain that 
$$\|z^j\|^2_p=\int_{\mathbb C}\frac{|z|^{2(j-\nu)}}{(1+|z|^{2a})^{\frac{p-\nu}{a}}}\,
\frac{a}{2\pi |z|^{2(1-a)}(1+|z|^{2a})^2}\,i\,dz\wedge d\overline z
=\int_0^{+\infty}\frac{2a\,r^{2(j-\nu)+2a-1}}{(1+r^{2a})^{\frac{p-\nu}{a}+2}}\,dr\;.$$
Using the substitution $r^{2a}=x/(1-x)$, it follows that 
\begin{equation}\label{e:zj}
\|z^j\|^2_p=B\left(1+\frac{j-\nu}{a}\,,1+\frac{p-j}{a}\right)<\infty\, \text{ if and only if }\, j_0\leq j\leq p\,.
\end{equation}
Since $(z^j,z^k)_p=0$ for $j\neq k$ we conclude that  
$$P_p(z)=\sum_{j=j_0}^p\frac{|z^j|^2_{h_p}}{\|z^j\|^2_p}\;,$$ 
which yields the desired formula for $P_p$. The remaining assertions are straightforward. 
\end{proof}

\begin{Remark}\label{R:cancel} If $\nu>0$ the function $\psi$ defined 
in \eqref{e:psi} is quasisubharmonic and has a pole at $0$ 
with Lelong number $\nu$. When $a=1$ and $0<\nu<1$, $\omega_1$ 
is the Fubini-Study metric on ${\mathbb P}^1$, hence it is smooth, 
and $P^{1,\nu}_p(z)\sim|z|^{-2\nu}$ blows up at $0$. 
So the presence of a logarithmic pole at $0$ in the Hermitian metric 
$h_p$ on $L^p$ makes $P_p(0)$ become infinite. 
On the other hand, if $\nu=0$ and $0<a<1$ then 
$P^{a,0}_p(0)\sim p/a$ while $P^{a,0}_p(z)\sim p$ for $z\neq 0$, 
by Theorem \ref{T:mt1}. So the presence of a conical singularity in $\omega_a$ 
at $0$ makes $P_p$ peak at $0$. 
Proposition \ref{P:rugby} shows that the following 
``interference" can appear in the presence of both a logarithmic pole at $0$ 
in $h_p$ and a conical singularity at $0$ in $\omega_a$: if $a\leq\nu<1$ 
then $P^{a,\nu}_p(0)=0$.
\end{Remark}

Next we let $|z|^{2a}=\frac{ay}{p}$ and, inspired by \cite{CCLW}, 
we are interested in the limit as $p\to\infty$ of the scaled Bergman kernel function

\begin{equation}\label{e:scale1}
F_p(y)=F^{a,\nu}_p(y):=\frac{1}{p}\,P^{a,\nu}_p\left(\left(\frac{ay}{p}\right)^{\frac{1}{2a}}\right),\,\;y\geq0\;.
\end{equation} 
Recall the definition of the Mittag-Leffler function 

\begin{equation}\label{e:MLfunc}
E_{r,s}(\zeta)=\sum_{j=0}^\infty\frac{\zeta^j}{\Gamma(rj+s)}\,,
\end{equation}
where $r>0$, $s\geq0$, and $\Gamma$ is the Euler Gamma function. 

\begin{Theorem}\label{T:rugby}
In the above setting we have that 
$$F_p(y)\to\frac{1}{a}\,y^{\frac{j_0-\nu}{a}}e^{-y}
E_{\frac{1}{a},1+\frac{j_0-\nu}{a}}\big(y^{1/a}\big)\,\text{ as } p\to\infty\,,$$
locally uniformly for $y\in(0,+\infty)$ (or for $y\in[0,+\infty)$ when $j_0-\nu\geq0$). 
\end{Theorem}

The theorem gives a scaling asymptotics of the Bergman kernel for $z$ 
approaching the singularity at 0 (resp. at $\infty$) as $p\to\infty$. 
Namely, if $|z|^{2a}=\frac{ay}{p}$ then 
\begin{equation}\label{e:scasym1}
P_p(z)\simeq\frac{p}{a}\,y^{\frac{j_0-\nu}{a}}e^{-y}
E_{\frac{1}{a},1+\frac{j_0-\nu}{a}}\big(y^{1/a}\big)\,\text{ as } p\to\infty\,.
\end{equation}
In particular, if $\nu=0$ we obtain that 
$$P_p(z)\simeq\frac{p}{a}\,e^{-y}E_{\frac{1}{a},1}
\big(y^{1/a}\big)\,\text{ as } p\to\infty\,,\,\text{ and }P_p(0)=\frac{p}{a}+1\,.$$

For the proof of Theorem \ref{T:rugby} we need the following lemmas. 

\begin{Lemma}\label{L:gamma}
If $r\geq0$ and $s\geq1$ then $\displaystyle\frac{\Gamma(r+s)}{\Gamma(s)}\leq 
e^{1/12}(r+s)^r$. Moreover, 
$\displaystyle\lim_{s\to\infty}\frac{\Gamma(r+s)}{\Gamma(s)s^r}=1$ 
locally uniformly in $r\geq0$. 
\end{Lemma}

\begin{proof} By Stirling's formula we have for all $x>0$ that 
$$\Gamma(x)=\sqrt{\frac{2\pi}{x}}\left(\frac{x}{e}\right)^xe^{\mu(x)}\,,\,
\text{ where }0<\mu(x)<\frac{1}{12x}\,.$$
Thus
$$\frac{\Gamma(r+s)}{\Gamma(s)}
=\sqrt{\frac{s}{r+s}}\;e^{-r}\left(1+\frac{r}{s}\right)^s(r+s)^re^{\mu(r+s)-\mu(s)}.$$
Since $s\geq1$, $\mu(r+s)-\mu(s)<1/12$. 
Moreover, $(1+r/s)^s\leq e^r$, so the inequality in the statement follows. Next
$$\frac{\Gamma(r+s)}{\Gamma(s)s^r}
=\sqrt{\frac{s}{r+s}}\;e^{-r}\left(1+\frac{r}{s}\right)^s
\left(1+\frac{r}{s}\right)^re^{\mu(r+s)-\mu(s)}\to1$$
as $s\to\infty$, locally uniformly for $r\geq0$.
\end{proof}

The next lemma is very simple and we omit its proof.

\begin{Lemma}\label{L:pser}
Let $f_n(\zeta)=\sum_{j=0}^\infty c_{n,j}\zeta^j$ be entire functions 
such that $c_{n,j}\to d_j$ as $n\to\infty$, for all $j\geq0$. 
Assume that there exist $\xi_j>0$ such that $|c_{n,j}|<\xi_j$ 
for all $n,\,j$, and the function $g(\zeta)=\sum_{j=0}^\infty\xi_j\zeta^j$ is entire. 
Then the function $f(\zeta)=\sum_{j=0}^\infty d_j\zeta^j$ is entire and 
$f_n\to f$ as $n\to\infty$ locally uniformly on $\mathbb C$.
\end{Lemma}

\begin{proof}[Proof of Theorem \ref{T:rugby}] We have 
\begin{eqnarray*}
F_p(y)&=&\left(1+\frac{ay}{p}\right)^{-\frac{p-\nu}{a}}
\frac{(ay)^{\frac{j_0-\nu}{a}}}{p^{1+\frac{j_0-\nu}{a}}}\,
\sum_{j=0}^{p-j_0}\frac{\Gamma\left(2+\frac{p-\nu}{a}
\right)(ay)^{\frac{j}{a}}}{\Gamma\left(1+\frac{j+j_0-\nu}{a}\right)
\Gamma\left(1+\frac{p-j-j_0}{a}\right)p^{\frac{j}{a}}}\\
&=&\frac{1}{a}\;y^{\frac{j_0-\nu}{a}}\left(1+\frac{ay}{p}\right)^{-\frac{p-\nu}{a}}G_p(y)\,,
\end{eqnarray*}
where
$$G_p(y)=\sum_{j=0}^{p-j_0}c_{p,j}y^{\frac{j}{a}}\;,\,\;c_{p,j}
=\frac{\Gamma\left(2+\frac{p-\nu}{a}\right)a^{1+\frac{j+j_0-\nu}{a}}}
{\Gamma\left(1+\frac{j+j_0-\nu}{a}\right)\Gamma\left(1+\frac{p-j-j_0}{a}\right)
p^{1+\frac{j+j_0-\nu}{a}}}\,.$$
By Lemma \ref{L:gamma} we have that for all $p,j$,
$$\frac{\Gamma\left(2+\frac{p-\nu}{a}\right)a^{1+\frac{j+j_0-\nu}{a}}}
{\Gamma\left(1+\frac{p-j-j_0}{a}\right)p^{1+\frac{j+j_0-\nu}{a}}}
\leq e^{1/12}\left(\left(2+\frac{p-\nu}{a}\right)\frac{a}{p}\right)^{1+\frac{j+j_0-\nu}{a}}<C^j\;,$$
for some constant $C>1$. Moreover, for $j$ fixed,
$$\lim_{p\to\infty}\frac{\Gamma\left(2+\frac{p-\nu}{a}\right)a^{1+\frac{j+j_0-\nu}{a}}}
{\Gamma\left(1+\frac{p-j-j_0}{a}\right)p^{1+\frac{j+j_0-\nu}{a}}}=1\,.$$
Hence
$$0<c_{p,j}<\frac{C^j}{\Gamma\left(1+\frac{j+j_0-\nu}{a}\right)}\;,\,\;
\lim_{p\to\infty}c_{p,j}=\frac{1}{\Gamma\left(1+\frac{j+j_0-\nu}{a}\right)}\;.$$
Now by Lemma \ref{L:pser}, 
$f_p(\zeta):=\sum_{j=0}^{p-j_0}c_{p,j}\zeta^j\to E_{\frac{1}{a},1+\frac{j_0-\nu}{a}}(\zeta)$ 
locally uniformly on $\mathbb C$. 
So $G_p(y)=f_p(y^{1/a})\to E_{\frac{1}{a},1+\frac{j_0-\nu}{a}}\big(y^{1/a}\big)$
locally uniformly in $y\geq0$, and the proof is complete. 
\end{proof}

We conclude this section with a closed formula for $P^{a,0}_p$ 
in the case when $a=1/s$ for some positive integer $s$. 

\begin{Proposition}\label{P:rugbyr} If $\nu=0$ and $a=1/s$, 
where $s>0$ is an integer, then
$$P_p(z)=P^{1/s,0}_p(z)=\left(p+\frac{1}{s}\right)
\left(1+\sum_{\ell=1}^{s-1}\left(\frac{1+e^{2\pi\ell i/s}|z|^{2/s}}{1+|z|^{2/s}}\right)^{ps}\right)\,.$$
In particular, $P_p(0)=sp+1$, while $P_p$ has the following asymptotic expansion 
on ${\mathbb C}\setminus\{0\}$\/{\rm:} for every $M>1$ 
there exists $\theta=\theta(M)>0$ such that if $1/M\leq|z|\leq M$ then 
$$P_p(z)=p+\frac{1}{s}+O(e^{-\theta p})\,.$$
\end{Proposition}

\begin{proof} By Proposition \ref{P:rugby},
$$P_p(z)=\frac{1}{(1+|z|^{2/s})^{ps}}
\sum_{j=0}^p\frac{|z|^{2j}}{B(1+sj,1+s(p-j))}
=\frac{sp+1}{(1+|z|^{2/s})^{ps}}\sum_{j=0}^p\binom{ps}{js}|z|^{2j}\,.$$
If $\zeta=e^{2\pi i/s}$ we have that 
$$\sum_{\ell=0}^{s-1}(1+\zeta^\ell y)^{ps}=
\sum_{k=0}^{ps}\binom{ps}{k}y^k\sum_{\ell=0}^{s-1}\zeta^{kl}
=s\sum_{j=0}^p\binom{ps}{js}y^{sj},$$
since $\zeta^k=1$ if $s$ divides $k$, and $\sum_{\ell=0}^{s-1}\zeta^{kl}=0$ 
if $s$ does not divide $k$. The conclusion follows if we let $y=|z|^{2/s}$ in the above formula.
\end{proof}
Thus the Bergman kernel $P^{1/s,0}_p$ has the same structure \eqref{e:Bexp003}
as in the case of the punctured disc endowed with the Poincar\'e metric,
namely only the first two terms are non-vanishing and the remainder
has exponential decay (cf.\ also Remark \ref{R:ga}). 

\subsection{Density of states on the lowest Landau level}\label{SS:dsLLL}

Here we comment on the relation of the results in the previous 
subsection to the density of states function on the lowest Landau level (LLL) 
on singular surfaces. The surface with the metric \eqref{e:oma} is pictured on Fig.\ 1. 
In physics terms \eqref{e:c2} means that the constant magnetic field 
(not pictured on Fig. 1) $B=p-\nu$ is turned on, with flux lines everywhere perpendicular to the surface.

\begin{wrapfigure}{l}{0.25\textwidth}
\centering
\begin{tikzpicture}
    \draw[scale=2,domain=-1:1,smooth,variable=\x,purple,very thick]  plot ({-(-(\x)^2+3)^.5},{\x});
        \draw[scale=2,domain=-1:1,smooth,variable=\x,purple,very thick]  plot ({(-(\x)^2+3)^.5-2.835},{\x});
                \draw[scale=1,domain=-.61:.63,dashed,variable=\x,purple,very thick]  
                plot ({\x-2.8},{(-(\x)^2+3)^.5-1.6});
                                \draw[scale=1,domain=-.65:.61,smooth,variable=\x,purple,very thick]  
                                plot ({\x-2.8},{-(-(\x)^2+3)^.5+1.6});
                                    \draw[black] (-2.54,2.4) node[below]{$\nu$};
                                    \draw[->,blue,very thick] (-2.83,1.7) -- (-2.84,2.3);
                                                                        \draw[black] (-2.84,-1.4) node[below]{$a$};
                                                                        \draw[black] (-2.84,1.8) node[below]{$a$};
  \end{tikzpicture}
  \vspace{0.5cm} \\
{\small Figure 1.\, Spindle with the cone angle $a$ and Aharonov-Bohm flux $\nu$.}\\
\end{wrapfigure}

In addition, there is a delta-function (Aharonov-Bohm) flux 
$\nu$ localized exactly at the north pole, see Fig.\ 1, so that the total 
flux of the magnetic field $p\in\mathbb Z$ 
through the compact surface is an integer. 
In terms of the singular metric \eqref{e:hp} 
the AB-flux $\nu$ is the Lelong number of the weight $\psi$ from \eqref{e:psi}.
This is the compact surface version of the setup of \cite{CCLW} 
where density of states was studied on a flat cone with a boundary. 

As was already pointed out in Remark \ref{R:cancel} the Bergman kernel \eqref{e:Pp} diverges as $P_p(z)\sim|z|^{2\nu}$ for $\nu<0$ and for 
$0\leq\{\nu\}<a$. This is because the Hermitian norm $|z^0|_{h^p}^2$ 
of the identity section $z^0$ \eqref{e:zj} corresponding to the 
LLL wave function with the smallest angular momentum, 
is singular at $z=0$ for these values of $\nu$, 
while the section is $L^2$-normalizable. 
The question may arise whether this section shall be kept in the spectrum. 
Here we answer this question in affirmative. The singular value of 
$P_p(0)$ is an artefact  of the delta-function form of the AB-flux, which should be smeared 
over some $\varepsilon$-neighborhood around zero. Hence the density is also smoothed out in this neighborhood. 

This leads to an interesting effect, when the AB-flux $\nu$ 
is allowed to vary over the real line, say, in the range $0\leq\nu<\infty$. 
Between $0\leq\nu<a$ the density $P_p(0)$ is peaked, 
while at $a\leq\nu<1$ it drops to zero $P_p(0)=0$. This pattern then repeats in the interval $[1,2]$ and so on. 
This is a manifestation of the Laughlin's ``shift register" on the LLL, 
first described in the original argument for the quantization 
of the Hall conductance \cite{L1}. As $\nu$ becomes greater than $a$, 
the identity section (wave function) $z^0$ becomes non-normalizable \eqref{e:zj} 
and drops from the spectrum of physical states, i.e. disappears into the conical singularity. 
At the same time a new $L^2$-normalizable section localized at the equator 
emerges so the that the total number of states is preserved. 
(In Laughlin's setting of annulus geometry, the wave functions travelled 
from the outer to the inner edge of the annulus, as AB-flux varied form $0$ to $1$.)

One consequence of Theorems \ref{T:mt1} and \ref{T:mt2} 
is that at a certain small distance (in units set by magnetic length $l_B^2\sim1/p$) 
away from the singular point the Bergman kernel tends to its constant value 
$2\pi P_p\sim p$. Thus the interesting behavior of the density profile 
happens around a small area near the singular point which shrinks 
as $p$ tends to infinity. One way to study the density profile, 
suggested in \cite{CCLW} is to use the rescaled coordinate 
$y=p|z|^{2a}/a$ in order to zoom in on the point $z=0$.
Remarkably, this leads to the universal finite result 
for the density profile near the conical singularity in Theorem \ref{T:rugby}, 
in agreement with the results of \cite{CCLW} for the flat cone.


\subsection{Metrics with a logarithmic pole}\label{SS:pole} 
We consider again ${\mathbb P}^1$ and the metric $\omega_a$ 
with conical singularities at $0$ and $\infty$, defined in \eqref{e:oma}. 
But here we endow the line bundle $L={\mathcal O}(1)$ 
with the Hermitian metric $h$ determined by the plurisubharmonic function 
\begin{equation}\label{e:phip}
\varphi(t,z)=\nu\log|z|+\frac{1-\nu}{2a}\,\log(|t|^{2a}+|z|^{2a})\,,\,\;0<a\leq1\,,\;0<\nu\leq1.
\end{equation}
We let $h_p=h^{\otimes p}$ be the induced metric on $L^p$. 
Note that $\varphi(1,z)$ has a logarithmic pole at $0$ with Lelong number $\nu$, and 
$$c_1(L^p,h_p)=p(1-\nu)\omega_a+p\nu\delta(0)\,.$$

Let $P_p=P^{a,\nu}_p$ be the Bergman kernel of the Hilbert space 
$H^0_{(2)}({\mathbb P}^1\setminus\{0,\infty\},L^p)$ 
of $L^2$-integrable holomorphic sections of $L^p$ relative to the metrics 
$h_p$ and $\omega_a$. By a similar calculation as in the proof of 
Proposition \ref{P:rugby} we obtain the following:

\begin{Proposition}\label{P:pole}
In the above setting, 
$$P^{a,\nu}_p(z)=\frac{|z|^{2(j_p-p\nu)}}{(1+|z|^{2a})^{\frac{p(1-\nu)}{a}}}\,
\sum_{j=0}^{p-j_p}\frac{|z|^{2j}}{B\left(1+\frac{j+j_p-p\nu}{a}\,,1
+\frac{p-j-j_p}{a}\right)}\,,\,\;z\in\mathbb C\,,$$
where $j_p=\floor{p\nu-a}+1$. We have that $j_p-p\nu\in(-a,1-a]$, 
and $j_p-p\nu=-\{p\nu\}$ when $a=1$. Moreover, if $\nu=1$ then $j_p=p$ and $P^{a,1}(z)=1$.
\end{Proposition}

Next we study the behavior of the scaled Bergman kernel function 
$F^{a,\nu}_p$ defined as in \eqref{e:scale1} by setting $|z|^{2a}=\frac{ay}{p}$. Namely:
$$F_p(y)=F^{a,\nu}_p(y):=\frac{1}{p}\,
P^{a,\nu}_p\left(\left(\frac{ay}{p}\right)^{\frac{1}{2a}}\right),\,\;y\geq0\;.$$
The difference with Section \ref{SS:rugby} is that now the sequence 
$\{F_p\}_{p\geq1}$ does not have a limit anymore, but it is relatively 
compact and its limit points are determined by the limit points of the bounded 
sequence $\{j_p-p\nu\}_{p\geq1}$. Note that when $\nu$ is irrational 
the latter sequence is dense in the interval $[-a,1-a]$.

\begin{Theorem}\label{T:pole}
In the above setting, assume that $p_k\to\infty$ is a sequence 
of positive integers such that $j_{p_k}-p_k\nu\to\theta\in[-a,1-a]$ as $k\to\infty$. Then
$$F_{p_k}(y)\to\frac{1-\nu}{a}\,((1-\nu)y)^{\frac{\theta}{a}}\,
e^{-(1-\nu)y}E_{\frac{1}{a},1+\frac{\theta}{a}}\big(((1-\nu)y)^{1/a}\big)\,,$$
locally uniformly for $y\in(0,+\infty)$. 
Hence the Bergman kernel 
has the following scaling asymptotics near $0$:  
$$P_{p_k}(z)\simeq\frac{p_k(1-\nu)}{a}\,((1-\nu)y)^{\frac{\theta}{a}}
\,e^{-(1-\nu)y}E_{\frac{1}{a},1+\frac{\theta}{a}}\big(((1-\nu)y)^{1/a}\big)\,,\,
\text{ where } |z|^{2a}=\frac{ay}{p_k}\,.$$
\end{Theorem}

\begin{proof} Arguing as in the proof of Theorem \ref{T:rugby}, 
we write 
$$F_p(y)=\frac{1}{a}\;y^{\frac{j_p-p\nu}{a}}
\left(1+\frac{ay}{p}\right)^{-\frac{p(1-\nu)}{a}}G_p(y)\,,$$
where
$$G_p(y)=\sum_{j=0}^{p-j_p}c_{p,j}y^{\frac{j}{a}}\;,\,\;c_{p,j}
=\frac{\Gamma\left(2+\frac{p(1-\nu)}{a}\right)a^{1+\frac{j
+j_p-p\nu}{a}}}{\Gamma\left(1+\frac{j+j_p-p\nu}{a}\right)
\Gamma\left(1+\frac{p-j-j_p}{a}\right)p^{1+\frac{j+j_p-p\nu}{a}}}\,.$$
By Lemma \ref{L:gamma} we have that for all $p,j$,
$$\frac{\Gamma\left(2+\frac{p(1-\nu)}{a}\right)a^{1+\frac{j+j_p-p\nu}{a}}}
{\Gamma\left(1+\frac{p-j-j_p}{a}\right)p^{1+\frac{j+j_p-p\nu}{a}}}
\leq e^{1/12}\left(\left(2+\frac{p(1-\nu)}{a}\right)
\frac{a}{p}\right)^{1+\frac{j+j_p-p\nu}{a}}<C_1^j\;,$$
for some constant $C_1>1$. By Stirling's formula, 
$\Gamma(x)>x^{-1/2}(x/e)^x$. It follows that 
$$0<c_{p,j}<C_2^j\,j^{-j/a},\;p\geq1\,,\;0\leq j\leq p-j_p\,,$$
with some constant $C_2>1$ (here $0^0:=1$). Using Lemma \ref{L:gamma} 
again, we conclude that for $j$ fixed,
$$\lim_{k\to\infty}c_{p_k,j}=\frac{(1-\nu)^{1+\frac{j+\theta}{a}}}
{\Gamma\left(1+\frac{j+\theta}{a}\right)}\;.$$
Now Lemma \ref{L:pser} implies that 
$$G_{p_k}(y)\to(1-\nu)^{1+\frac{\theta}{a}}E_{\frac{1}{a},
1+\frac{\theta}{a}}\big(((1-\nu)y)^{1/a}\big)\,,$$ 
as $k\to\infty$, locally uniformly in $y\geq0$. This completes the proof. 
\end{proof}


\end{document}